\documentclass[a4paper,fleqn]{cas-sc}

\usepackage[authoryear,longnamesfirst]{natbib}
\usepackage{amssymb}
\usepackage{nccmath, mathtools}
\usepackage{graphicx}
\graphicspath{{Figures/}}
\usepackage{amsmath,bm}
\usepackage{amssymb}
\usepackage{multicol}
\usepackage{tikz}
\usetikzlibrary{arrows}
\usepackage{pgfplots}
\pgfplotsset{compat=1.15}
\usepackage{subfigure}
\usepackage{float}

\def\tsc#1{\csdef{#1}{\textsc{\lowercase{#1}}\xspace}}
\tsc{WGM}
\tsc{QE}

\newtheorem{theorem}{Theorem}
\newtheorem{lemma}[theorem]{Lemma}
\newdefinition{remark}{Remark}
\newproof{proof}{Proof}

\begin{document}
\let\WriteBookmarks\relax
\def\floatpagepagefraction{1}
\def\textpagefraction{.001}

\shorttitle{AMR Control}    

\shortauthors{A. Peterson,  J. Romero, P. Aguirre, K. Achraya, B. Nasri}  

\title [mode = title]{Assessing the Impact of Mutations and Horizontal Gene Transfer on the AMR Control: A Mathematical Model}  



\author[1]{ Alissen Peterson}
\ead{alissen.peterson@sansano.usm.cl }
\credit{Conceptualization, Data Curation, Formal Analysis, Writing - Original Draft Preparation, Visualization}
\affiliation[1]{organization={Departamento de Matemática, Universidad Técnica Federico Santa María},
            city={Valparaiso},
            country={Chile}}

\author[2]{ Jhoana P. Romero-Leiton}[orcid=0000-0002-2788-178X]

\ead{jhoana.romero@umanitoba.ca}
\credit{Conceptualization, Data Analysis, Methodology, Formal Analysis, Writing- Review and Editing, Supervision, Resources }

\affiliation[2]{organization={Department of Mathematics, University of Manitoba},
            city={Winnipeg},
            country={Canada}}

\author[1]{ Pablo Aguirre}

\ead{pablo.aguirre@usm.cl}
\credit{Conceptualization, Methodology, Formal Analysis, Writing- Review and Editing, Supervision, Resources}

\author[4]{ Kamal R. Acharya}[orcid=0000-0001-6707-3536]
\ead{kamal.acharya@umontreal.ca  }
\credit{Conceptualiation, Methodology, Writing - Review and Editing  }

\affiliation[4]{organization={
Département de Médecine
Sociale et Préventive, École de
Santé Publique},
            city={Montreal},
            country={Canada}}

\author[4]{ Bouchra Nasri\corref{cor1}}[orcid=0000-0001-6334-3105]
\cortext[cor1]{Corresponding author.}
\ead{bouchra.nasri@umontreal.ca }
\credit{Conceptualiation, Methodology, Writing - Review and Editing}


\begin{abstract}
Antimicrobial resistance (AMR) poses a significant threat to public health by increasing mortality, extending hospital stays, and increasing healthcare costs. It affects people of all ages and affects health services, veterinary medicine, and agriculture, making it a pressing global issue. Mathematical models are required to predict the behaviour of AMR and to develop control measures to eliminate resistant bacteria or reduce their prevalence. This study presents a simple deterministic mathematical model in which sensitive and resistant bacteria interact in the environment, and mobile genetic elements (MGEs) are functions that depend on resistant bacteria. We analyze the qualitative properties of the model and propose an optimal control problem in which avoiding mutations and horizontal gene transfer (HGT) are the primary control strategies. We also provide a case study of the resistance and multidrug resistance (MDR) percentages of \textit{Escherichia coli} to gentamicin and amoxicillin in some European countries using data from the European Antimicrobial Resistance Surveillance Network (EARS-Net). Our theoretical results and numerical experiments indicate that controlling the spread of resistance in southern European regions through the supply of amoxicillin  is challenging. However, the host immune system is also critical for controlling AMR.
\end{abstract}


\maketitle

\section{Introduction}\label{sec-intro}
In 1928, Alexander Fleming introduced penicillin. To date, this antimicrobial has been widely used in the medical world, causing a significant impact owing to its high percentage of bacterial death. Antimicrobials have entirely revolutionized the world since then by prolonging human life (even in animals and plants) because a simple flu or a complex disease are no longer a certain death \cite{newell2010food}. However, since the creation of the first antimicrobial, the human population has abused their benefits, causing pathogens to mutate more quickly, thus creating antimicrobial resistance (AMR).
\par
AMR occurs when pathogens develop the ability to defeat drugs designed to kill them \cite{mechanisms}. This means that the pathogen is not killed and continues to grow. Infections caused by AMR pathogens are difficult or sometimes impossible to treat \cite{about_antimicrobial}. AMR infections usually require extended hospital stays and use of expensive and toxic alternatives.
\par 
How do pathogens become resistant to antimicrobials? There are two mechanisms of resistance acquisition \cite{mechanisms}:   (a)  resistance by mutations,  in which a subset of pathogens derived from a susceptible population develops mutations in genes that affect the activity of the drug, resulting in the survival of the pathogen in the presence of antimicrobial molecules. (b) Horizontal gene transfer: This scenario involves acquiring foreign DNA material that encodes resistance determinants. Classically, pathogens acquire external genetic material through three main strategies: i) transformation (incorporation of naked DNA), ii) transduction (phage-mediated), and iii) conjugation (pathogenic “sex”). The last method uses mobile genetic elements (MGEs) as vehicles to share valuable genetic information. The essential MGEs are plasmids and transposons, which play crucial roles in developing and disseminating AMR in clinically relevant organisms \cite{munita2016mechanisms}.
\par
AMR is a serious public health problem worldwide. This is a global problem with geographical variation. Acquired AMR is common in apparently healthy isolates. Several pathogens, such as {\it Vibrio cholerae}, {\it Shigella flexneri}, {\it Streptococcus
pneumoniae}, and {\it Escherichia coli}, are increasingly developing resistance, particularly to low-cost first-line broad-spectrum antimicrobials. Introducing new drugs (e.g., Fluoroquinolones) has been relatively rapid, followed by the emergence and dissemination of resistant strains \cite{hoge1998trends}. Outbreaks occur as resistance develops, which may result in high mortality rates. 
\par
Since AMR is a natural response of pathogens to an exposure to antimicrobials, an effective control strategy has to be one of the containment strategies aimed at reducing the rate of emergence and spread of resistance. According to \cite{komolafe2003antibiotic}, these goals can be achieved using four control strategies:  decrease in selective pressure, adoption of reasonable infection control, surveillance of antibiotic resistance, and increase in research activities. 
\par
Numerous mathematical models have been developed with the aim of comprehending the AMR issue through a mathematical lens. Some of the most recent studies help to understand the dynamics of transmission and spread of AMR \cite{ibarguen2013interactions, ibarguen2014mathematical, mostefaoui2014mathematical, ternent2015bacterial, jin2015mathematical, dacsbacsi2017fractional, merdan2017comparison, esteva2018modeling, birkegaard2018send}. In contrast, others have focused on the influence of the immune response \cite{dacsbacsi2016mathematical, mondragon2016simple, dacsbacsi2016dynamics, ibarguen2018mathematical, dacsbacsi2018analysis} and on the acquisition of resistance by plasmids \cite{ibarguen2016mathematical, ibarguen2019stability, leclerc2019mathematical}. Regarding AMR control, only a few studies have established analytical results for the optimal control of an infectious disease under drug resistance (see, for example, \cite{BLL97, SPH98, PBK15, CCE19, LLY17, romero2022optimal}.
\par
Instead of considering a population model for hosts infected with a microbial disease, in this study, we simplify the mathematical model formulated by \cite{ibarguen2019stability} to propose a simple optimal control problem for bacterial resistance to antibiotics. In particular, we assumed that bacteria interact under resistance by mutations and HGT (by plasmids and transposons). We first analyze the qualitative properties of the model. We then prove the existence of controls and present a case study in some European countries using data reported by the European Antimicrobial Resistance Surveillance Network (EARS-Net) \cite{european2010european} for \textit{Escherichia coli} bacteria. Our findings show that it is possible to significantly decrease the spread of resistance in geographic regions of Europe with low rates of resistance and multidrug resistance (MDR). In addition, one of the most significant results of this study shows that the immune system plays a fundamental role in decreasing the spread of resistant bacteria.
\par
The remainder of this paper is organized as follows. The model is introduced in Section \ref{sec-the model}. Section \ref{sec:qualitative} presents an analytical study of the model’s dynamics. The optimal control problem is discussed in Section \ref{sec:control}. The case of study is treated in Section \ref{sec:case} Finally, Section \ref{sec:discussion} presents conclusions, open questions, and future work.

\section{The model}\label{sec-the model}
\cite{ibarguen2019stability} proposed a mathematical model to describe the spread of resistance in bacteria populations, considering plasmids and antibiotic concentration as dynamic variables. In this model, we assume that $S(t)$ and $R(t)$ represent the population of sensitive and resistant bacteria, respectively, at time $t$, where $\textbf{X}(t)=(S(t), R(t))$ is the state vector. The growth of sensitive and resistant bacteria follows logistic dynamics with carrying capacity $K$ and birth rates $\beta_S$ and $\beta_R$, respectively. The model assumes that MGEs are a function of the population of resistant bacteria, represented by $P(R)=aR^{n}$, where $a$ is a constant of proportionality and $n\geq 1$. Antibiotics are assumed to be administered at a constant rate $\Lambda$, with $\bar{\alpha}$ representing the elimination ratio of sensitive bacteria by antibiotics, $\bar{\gamma}$ representing the elimination rate of bacteria by the immune response, and $\bar{q}$ corresponding to the mutation ratio of sensitive bacteria by antibiotics. 
\par 
Preventing natural and acquired mutations by MGEs, strategies to prevent the acquisition of resistance, are introduced as  control variables. The study uses these measures to minimize the cost function. The control variables $h_1$ and $h_2$ represent the control of natural mutations and the acquisition of MGEs, respectively, and take values between $0$ and $1$. If the measure is ineffective, $h_i=0$ for $i=1,2$; if it is fully effective, $h_i=1$ for $i=1,2$. One objective of the study is to minimize the number of resistant bacteria. Therefore, the model defines an optimal control problem as follows: 
\begin{equation}\label{model}
    \left\{ \begin{array}{ll}
              &J[\textbf{h}]=\displaystyle\int_0^T\left[cR+(w_1+w_2h_1)h_1+(b_1+b_2h_2)h_2 \right] dt, \\ \\
              &\dfrac{dS}{dt}=  \beta_SS\left(1-\dfrac{S+R}{K}\right)-(\bar\alpha\Lambda+\bar \gamma)S-(1-h_1(t))\bar q\Lambda S-(1-h_2(t)) aRS, \\ \\
              &\dfrac{dR}{dt}= \beta_RR\left(1-\dfrac{S+R}{K}  \right)+(1-h_1(t))\bar q\Lambda S +(1-h_2(t)) aRS-\bar\gamma R,\\ \\
              &\textbf{X}(0)=(S(0), R(0))=\textbf{X}_0 \\  \\
              &\textbf{X}(T)=(S^*, R^*)=\textbf{X}_1.
            \end{array} \right. 
\end{equation}

In the above formulation, $J[\textbf{h}]$ represents the value of the net profit of resistant bacteria. Since $h_1$ and $h_2$ can be interpreted as controls by education, $cR+(w_1+w_2h_1)h_1+(b_1+b_2h_2)h_2 $ is the cost by education. We use quadratic forms in the functional because they are a traditional form in simple linear systems and often used in biological models (see, e.g., \cite{romero2022optimal, optimal}).
 Additionally, we assume that control variables are in the set
$$
\mathcal{U}=\left\{h(t): h(t) \; \text{is Lebesgue measurable and} \; 0\leq h(t)\leq 1, t\in[0,T]\right\},
$$
called \textit{the set of admissible controls}.
\par
All the parameters in Model \eqref{model} are positive and listed in Table \ref{table_1}. We make use of the same assumption as presented in \cite{ibarguen2019stability}  that the growth rate of resistant bacteria is either equal to or lower than that of sensitive bacteria, namely $\beta_R\leq \beta_S$.

\begin{table*}[h]
	\caption{Description and units of the Model \eqref{model} parameters.}
	\label{table_1}
		\begin{tabular*}{\tblwidth }{@{}LLL@{}}
		\toprule
			Parameter   &     Description &Dimension   \\
\midrule
$\beta_S$  & Growth rate of sensitive bacteria  & $1/time$   \\
$\beta_R$  & Growth rate of resistant bacteria  &$1/time$   \\
$\bar\alpha$& Elimination percentage of sensitive bacteria by antibiotics  &  Dimensionless   \\
$1/\Lambda$   &  Frequency of taking the medication   & Time   \\
$\bar\gamma$ &Elimination rate of bacteria by the host immune response & $1/time$   \\
$\bar q$   &Mutation percentage of sensitive bacteria by antibiotics  & Dimensionless  \\
$a$        &Constant of proportionality &$1/ (pop\times time)$  \\
$K$        &Carrying capacity of bacteria & Dimensionless  \\
		\bottomrule
	\end{tabular*}
\end{table*}

\section{Qualitative behaviour of the model} \label{sec:qualitative}
In order to analyze the impact of preventive measures on the dynamics of Model \eqref{model}, we first assume that $h_1(t)$ and $h_2(t)$ are constant. To simplify the analysis, we introduce the following change of variables:

 \begin{equation}\label{new_variables}
    x=\frac{S}{K}, \quad y=\frac{R}{K}, \quad \text{and} \quad \tau=aKt,
 \end{equation}
 
 with the parameters
 \begin{equation}\label{new_parameters}
    \begin{array}{lllll}
      \beta_s=\dfrac{\beta_S}{aK} & \beta_r=\dfrac{\beta_R}{aK} & q=\dfrac{\bar q\Lambda}{aK} 
      \alpha=\dfrac{\bar\alpha\Lambda}{aK} &\gamma=\dfrac{\bar\gamma}{aK}.  &
    \end{array}
 \end{equation}
 
Therefore, System \eqref{model} can be expressed in dimensionless form as follows:

 \begin{equation}\label{model_res}
    \left\{ \begin{array}{ll}
              \dot x= & \beta_sx[1-(x+y)]-(\alpha+\gamma) x-(1-h_1)qx-(1-h_2)xy \\ \\
              \dot y=&\beta_ry[1-(x+y)]+ (1-h_1)qx+(1-h_2)xy-\gamma y.
            \end{array} \right.
 \end{equation}
The following lemma states the well-posedness of System \eqref{model_res} by establishing that solutions are both non-negative and bounded. Specifically, we establish that $\Omega$, depicted in Figure \ref{fig_invarianta}, is a trapping region, meaning that the orbits of \eqref{model_res} with initial conditions in $\Omega$ will remain within $\Omega$ for all $t\geq 0$.

\begin{figure}[h]
    \centering
\includegraphics[width=8.1 cm, height=6cm]{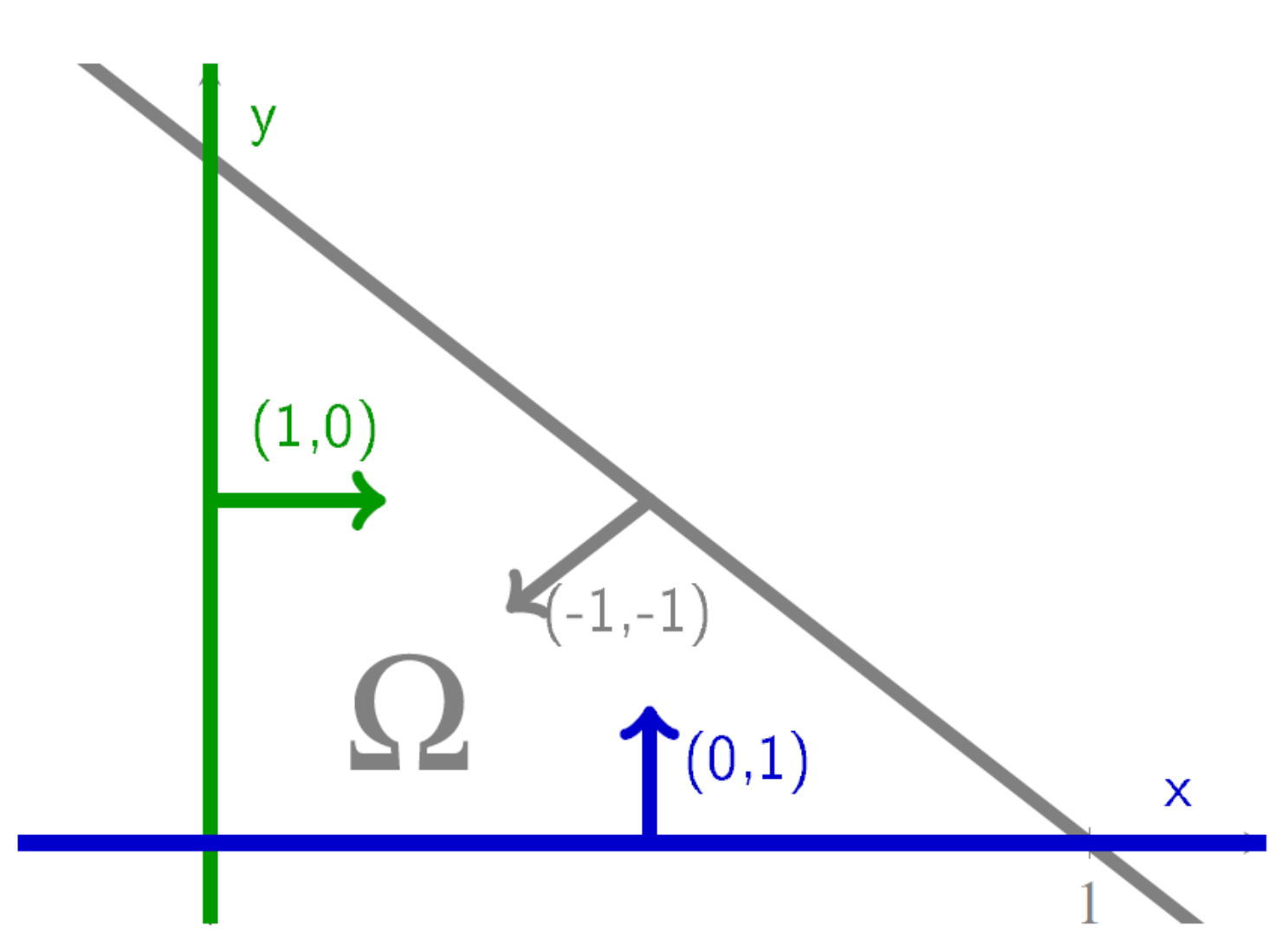}
\caption{Region of biological interest for System \eqref{model_res}. It can be defined as  $\Omega={(x,y)\in\mathbb{R}^2:\ \ y\leq -x+1, \ 0\leq x\leq 1, \ 0\leq y\leq 1}$ in the mathematical context.}
    \label{fig_invarianta}
\end{figure}

\begin{lemma}\label{lemma1}
The set $\Omega$ is a trapping region for System \eqref{model_res}.
\end{lemma}

The proof is available in Appendix \ref{appendix:A1}. \\
\par 
In order to give a biological meaning to the conditions that ensure the existence and stability of the equilibrium points of System \eqref{model}, we introduce the following epidemiological thresholds:

\begin{equation}\label{thresholds}
    \begin{array}{llll}
      R_r= \dfrac{\beta_r}{\gamma},              &  R_s=\dfrac{\beta_s}{\alpha+\gamma+q(1-h_1)}, &      \\ \\
      h_s=\dfrac{1-h_2}{\alpha+\gamma+q(1-h_1)}, & h_{r1}=\dfrac{1-h_1}{\gamma},              &h_{r2}=\dfrac{1-h_2}{\gamma}.
    \end{array}
\end{equation}
The classic definition of the \textit{Basic Reproductive Numbe}r, denoted by $\mathcal{R}0$, can be used to interpret the thresholds mentioned above. Specifically, $R_r$ represents the number of bacteria produced by the fraction of resistant bacteria that evade the immune response, while $R_s$ represents the number of bacteria produced by the fraction of susceptible bacteria that do not have natural mutations and are not eliminated by the antibiotic or immune system. Moreover, $h_s$ represents the number of mutations due to uncontrolled horizontal gene transfer since the bacteria do not have natural mutations and are not eliminated by the antibiotic or immune system. Additionally, $h_{r1}$ and $h_{r2}$ represent the number of uncontrolled mutations (natural or acquired) because the bacteria are not attacked by the immune response. By using \eqref{thresholds}, System \eqref{model_res} can be equivalently expressed as:

 \begin{equation}\label{model_res2}
    \left\{ \begin{array}{ll}
              \dot x= & x\left[\dfrac{\gamma h_{r2}}{h_s}[R_s(1-(x+y))-1-h_s y]\right]:=f_1(x,y),\\ \\
              \dot y=&\gamma[R_r y(1-(x+y))+h_{r1}qx+h_{r2}xy-y]:=f_2(x,y).
            \end{array} \right.
 \end{equation}


\subsection{Existence of equilibrium points}
In this section, we explore the existence of  equilibrium points for System \eqref{model_res2} by solving the following system of algebraic equations:

\begin{equation}\label{eq_equilibria}
\begin{array}{ll}
0=& x\left[ \dfrac{\gamma h_{r2}}{h_s}\left[R_s(1-(x+y))-1-h_s y] \right]\right]  \\ \\
0=&\gamma \left[R_r y(1-(x+y))+h_{r1}qx+h_{r2}xy-y \right].
\end{array}
\end{equation}

After performing some algebraic manipulations in \eqref{eq_equilibria}, it has been determined that System \eqref{model_res2} possesses three equilibrium points. These equilibrium points include the trivial equilibrium denoted as $\textbf{P}_0 = (0,0)$, another equilibrium point where resistant bacteria persists denoted as $\textbf{P}_1= \left(0, \dfrac{R_r-1}{R_r} \right)$, and a coexistence equilibrium point denoted as $\textbf{P}^*=\left(\frac{R_s-1}{R_s}-\left[\frac{R_s+h_s}{R_s} \right]y^*, y^*\right)$, where $R_s$, $R_r$, and $h_s$ are the thresholds defined in \eqref{thresholds} and $y^*$ will be defined in \eqref{y}. 
\par 
The conditions for the existence of these equilibrium  points are summarized in the following theorem.

\begin{theorem}\label{teoexistencia} 
For System \eqref{model_res2} in 
$\Omega$, the equilibrium points  are given by $\textbf{P}_0=(0,0)$, which always exists;  $\textbf{P}1=\left(0, \dfrac{R_r-1}{R_r} \right)$, which exists if and only if $R_r>1$; and 
$\textbf{P}^*=(x^*, y^*)$, where $x^*$ and $y^*$ are defined by \eqref{2.14} and \eqref{y}, respectively. The equilibrium point $\textbf{P}^*$ exists if either of the following conditions are satisfied:
\begin{enumerate}
\item $R_s>1$, $0<h_1<1$, $0<h_2<1$, and $R_r (h_s+1)< h_s +R_s$.
\item $R_s>1$, $h_1=1$, $0<h_2<1$, and $h{r2}(R_s-1)+R_r - R_s>0$.
\end{enumerate}
 If $h_2=1$, the coordinates of $\textbf{P}^*=(x^*, y^*)$ are given by \eqref{2.14} and $y^*=\dfrac{a_0}{a_1}$, respectively, with $a_0$ and $a_1$ defined on \eqref{element_quadratic_y}, provided $R_s>1$, $0<h_1<1$  and  $-h_{r1}q(R_s+h_s)+R_r - R_s>0$.
\end{theorem}

\begin{proof}
For one hand, if  $x=0$   in both equations of \eqref{eq_equilibria} we have that either $y=0$ or
$y=\dfrac{R_r-1}{R_r}$.
Thus, we obtain the trivial equilibrium $\textbf{P}_0 = (0,0)$, and an equilibrium point  where resistant bacteria persist $\textbf{P}_1= \left(0, \dfrac{R_r-1}{R_r} \right)$ which exists only if $R_r>1$.
\noindent
On the other hand, if $x\neq 0$ the equations \eqref{eq_equilibria} can be rewritten as

\begin{equation}\label{eq_equilibria3}
 \begin{array}{ll}
& R_s-1-R_sx-[R_s+h_s]y =0  \\ \\
&-R_ry^2+[h_{r2}-R_r]xy+(R_r-1)y+h_{r1}qx=0.
\end{array}
\end{equation}

\noindent 
From the first equation of \eqref{eq_equilibria3} we obtain
\begin{equation}\label{eq000}
    \frac{R_s-1}{R_s}=x+\left[\frac{R_s+h_s}{R_s} \right]y.
\end{equation}
\noindent
It follows that a necessary condition for the existence of an equilibrium with susceptible and resistant bacteria is $R_s>1$.  By solving for $x$ in \eqref{eq000}, we obtain

\begin{equation}\label{2.14}
    x= \frac{R_s-1}{R_s}-\left[\frac{R_s+h_s}{R_s} \right]y.
\end{equation}
\noindent
Therefore, $x$ defined above is positive if and only if the following condition is fulfilled:
\begin{equation}\label{ymax}
    y< y_{max},  \; \; \text{where} \;\; y_{max}=\frac{R_s-1}{R_s+h_s}.
\end{equation}
\noindent
Now, replacing the expression $R_s-1=(R_s+h_s)y_{max}$ in  \eqref{2.14},  we obtain
\begin{equation}\label{2.17}
    x^*=\frac{(R_s+h_s)(y_{max}-y)}{R_s}.
\end{equation}
Substituting \eqref{2.17} into the second equation in \eqref{eq_equilibria3} and using \eqref{thresholds} we obtain the following quadratic equation for the variable $y$:
\begin{equation}\label{quadratic_y}
   p(y)= -a_2y^2+a_1y+a_0=0,
\end{equation}
where

\begin{equation}\label{element_quadratic_y}
\begin{array}{rcl}
a_{2}&=&h_{r2}(R_s+h_s)-h_sR_r  \\
   &=&\left(\dfrac{1-h_2}{\gamma}\right)\left(\dfrac{\beta_s+(1-h_2) }{\alpha+\gamma+q(1-h_1)}\right)-\left(\dfrac{1-h_2}{\alpha+\gamma+q(1-h_1)}\right)\left(\dfrac{\beta_r}{\gamma}\right)\\
   &=&\dfrac{h_{r2}}{\alpha+\gamma+q(1-h_1)}(\beta_s-\beta_r)+h_{r2}h_s,\\ \\
 a_1&=&(R_s+h_s)\left[(h_{r2}-R_r)y_{max}-h_{r1}q\right]+R_s(R_r-1) \\
   &=&(R_s+h_s)\left[(h_{r2}-R_r)\dfrac{R_s -1}{R_s + h_s} -h_{r1}q\right]+R_{s} (R_r-1) \\
   &=&(h_{r2}-R_r)(R_s-1)-h_{r1}q(R_s+h_s)+R_sR_r-R_s \\ 
   &=&h_{r2}(R_s-1)-R_r(R_s-1) -h_{r1}q(R_s+h_s)+R_sR_r -R_s\\
   &=&h_{r2}(R_s-1)-h_{r1}q(R_s+h_s)+R_r - R_s,\\\\
 a_0&=&h_{r1}q(R_s+h_s)y_{max}  \\
   &=&h_{r1}q(R_s+h_s)\dfrac{R_s - 1}{R_s+h_s}.\\
   &=&h_{r1}q(R_s-1).
\end{array}
\end{equation}

Assuming $R_{s} > 1$ we have that $a_0$ is non-negative (in particular, if $h_{r1}>0$ ---or equivalently, $h_1<1$--- then $a_0>0$).  Similarly, provided $h_{r2}>0$ ---or equivalently, $h_2<1$--- we have that $a_2>0$ since $\beta_s>\beta_r$. In such case \eqref{quadratic_y} has only one positive root: 
\begin{equation}\label{y}
  y^*=\frac{a_1 + \sqrt{a_1^2+4a_2a_0}}{2a_2}.   
\end{equation}
Hence the positive equilibrium point is located at $(x^*,y^*)$ where $x^*$ is given in \eqref{2.14} and $y^*$ is as in \eqref{y}, provided $y^*<y_{max}$ given in \eqref{ymax}.
\par 
To prove that $y^*<y_{max}$, we consider again the polynomial \eqref{quadratic_y}
with coefficients \eqref{element_quadratic_y}. Note that  $p(y)$ is a concave parabola.  Since $y^*>0$ and $p(0)=a_0$, to make sure that $y^*<y_{max}$, it is enough to show that $p(y_{max})<0$ (see Appendix \ref{appendix:A2}). 
\par 
Note that if $h_{r2}>0$ and $h_{r1}=0$ (or equivalently $h_1=1$), we have $a_0=0$. In this case, the positive equilibrium is still given by $(x^*, y^*)$ provided $h_{r2}(R_s-1)+R_r - R_s>0$. Additionally, if $h_{r2}=0$, we have $a_2=0$ and $y^*=\dfrac{a_0}{a_1}>0$ provided $h_1<1$ and  $-h_{r1}q(R_s+h_s)+R_r - R_s>0$.
\end{proof}

\subsection{Stability of the equilibrium points}
In order to assess the stability of the equilibrium points of System \eqref{model_res2}, the linearization of the vector field defined by the right-hand side of \eqref{model_res2} is utilized. At an equilibrium point $\textbf{P}=(x,y)$, the Jacobian matrix $\textbf{J}$ of \eqref{model_res2} is defined as follows:

\begin{equation}\label{Jacobian_general}
    \textbf{J}(\textbf{P})=\left(
                             \begin{array}{cc}
                             \dfrac{\gamma h_{r2}}{h_s}[R_s-1-2R_sx-(R_s+h_s)y]   & -\dfrac{\gamma h_{r2}}{h_s}(R_s+h_s)x  \\ \\
                              \gamma[h_{r1}q+(h_{r2}-R_r)y]  &\gamma[R_r-1-(R_r-h_{r2})x-2R_ry]
                             \end{array}
                           \right).
\end{equation}

By applying the matrix defined in \eqref{Jacobian_general} to each of the three equilibrium points defined in Theorem \ref{teoexistencia}, we can establish the following outcome.

\begin{theorem}\label{teo_stability}
Let us consider System \eqref{model_res2} with the equilibrium points state in Theorem \ref{teoexistencia}.  Therefore
\begin{enumerate}
    \item The origin $\textbf{P}_0=(0,0)$ is locally asymptotically stable (\textit{LAS}) in $\Omega$ if and only if $R_s<1$ and $R_r<1$.
    \item The equilibrium $\textbf{P}_1=\left(0, \dfrac{R_r-1}{R_r}  \right)$ is \textit{LAS} in $\Omega$ if $R_r>1$ and $R_r>\dfrac{R_s +h_s}{1+hs}$.
    \item The equilibrium $\textbf{P}^*=(x^*,y^*)$ is\textit{ LAS} in $\Omega$ if $R_s>1$, $0<h_1<1$, $0<h_2<1$ and $R_r<\dfrac{R_s+h_s}{(h_s+1)}$.
\end{enumerate}
\end{theorem}

\begin{proof}
\begin{enumerate}
\item The matrix $\textbf{J}$ evaluated in $\textbf{P}_0=(0, 0)$ is
\begin{equation}\label{Jacobian_p0}
    \textbf{J}(\textbf{P}_0)=\left(
                             \begin{array}{cc}
                            \dfrac{\gamma h_{r2}}{h_s}(R_s-1)   & 0 \\ \\
                              \gamma h_{r1}q  &\gamma(R_r-1)
                             \end{array}
                           \right),
\end{equation}
whose eigenvalues are $\lambda_1= \dfrac{\gamma h_{r2}}{h_s}(R_s-1)$ and $\lambda_2=\gamma(R_r-1)$.  Thus, the origin is \textit{LAS} if conditions $R_s<1$ and $R_r<1$ are satisfied. 
\item $\textbf{J}$  evaluated in $\textbf{P}_1$ is given by

\begin{equation}\label{Jacobian_p1}
    \textbf{J}(\textbf{P}_1)=\left(
                             \begin{array}{cc}
                             \dfrac{\gamma h_{r2}}{h_s}\left(R_s-1-\dfrac{(R_s+h_s)(R_r-1)}{R_r}\right)   & 0 \\ \\
                              \gamma \left( h_{r1}q +\dfrac{(h_{r2}-R_r)(R_r-1)}{R_r}\right) &-\gamma (R_r-1)
                             \end{array}
                           \right).
\end{equation}

The eigenvalues of $\textbf{J}(\textbf{P}_1)$ are $\lambda_1 =\dfrac{\gamma h_{r2}}{h_s}\left(R_s-1-\dfrac{(R_s+h_s)(R_r-1)}{R_r}\right)$ and $\lambda_2=-\gamma (R_r-1)$.  Note that $\lambda_2<0$ if $R_r>1$.
Additionally, $\lambda_1<0$  if
$$
\begin{array}{cl}
 &R_s-1-\dfrac{(R_s+h_s)(R_r-1)}{R_r}<0 \\
 \Leftrightarrow  &\dfrac{R_s-1}{R_s+h_s}<\dfrac{R_r-1}{R_r} \\
  \Leftrightarrow &y_{max}<\dfrac{R_r-1}{R_r} \\
  \Leftrightarrow &R_r>\dfrac{1}{1-y_{max}},
\end{array}
$$
where $y_{max}$ is given in \eqref{ymax}.
Thus, equilibrium $\textbf{P}_1$ is \textit{LAS} if $R_r>1$, $R_s>1$ and $R_r>\dfrac{1}{1-y_{max}}$.
\item To determine conditions for the stability of $\mathbf{P^*}$,  let us  denote $J_{ij}$ to the $(i,j)$-entry in \eqref{Jacobian_general}. From the first equation of \eqref{eq_equilibria3},  we have $R_s-1-R_sx-(R_s + h_s)y = 0$.  Thus, 
\begin{equation}\label{j11}
 J_{11}(\mathbf {P})  =\frac{\gamma h_{r2}}{h_s}[R_s-1-2R_s x^{*}-(R_s+h_s)y^{*}] 
            =-\frac{\gamma h_{r2}}{h_s}R_s x.
\end{equation}
From the second equation of  \eqref{eq_equilibria3}, we have $(h_{r2}-R_r)x=R_r y -(R_r-1)-h_{r1}q\frac{x}{y}$.  Substituting this value in $J_{22}$ we obtain

\begin{equation}\label{J22}
 J_{22}(P)  =\gamma[R_r-1-(R_r-h_{r2})x-2R_ry]
            = -\gamma[h_{r1}q\frac{x}{y}+R_ry].
\end{equation}

Replacing  \eqref{j11} and \eqref{J22} in \eqref{Jacobian_general},  it can be re-written as
\begin{equation}\label{Jacobian_general_con_mod}
    \textbf{J}(\textbf{P})=\left(
                             \begin{array}{cc}
                             -\dfrac{\gamma h_{r2}}{h_s}R_sx & -\dfrac{\gamma h_{r2}}{h_s}(R_s+h_s)x  \\ \\
                              \gamma[h_{r1}q+(h_{r2}-R_r)y]  &-\gamma[h_{r1}q\frac{x}{y}+R_ry]
                             \end{array}
                           \right).
\end{equation}
For $0<h_2<1$ this Jacobian matrix $\textbf{J}$ evaluated in $\textbf{P}^*=(x^*,y^*)$ stated in Theorem~\ref{teoexistencia}
is:
\begin{equation}\label{Jacobian_general1}
\textbf{J}(\textbf{P}^*)=\left(
                             \begin{array}{cc}
                             -\dfrac{\gamma h_{r2}}{h_s}R_sx^{*} & -\dfrac{\gamma h_{r2}}{h_s}(R_s+h_s)x^{*}  \\ \\
                              \gamma[h_{r1}q+(h_{r2}-R_r)y^{*}]  &-\gamma\left[h_{r1}q\dfrac{x^{*}}{y^{*}}+R_ry^{*}\right]
                             \end{array}
                           \right).  
\end{equation}  

The characteristic polynomial of $\textbf{J}(\textbf{P}^*)$ is given by
$$|J(P)-I\lambda|=\lambda^2-\tau\lambda+\delta,$$
where
\begin{align*}
    &\tau=Trace= J_{11}(\textbf{P}^*)+J_{22}(\textbf{P}^*),\\
    &\delta =Determinant= J_{11}(\textbf{P}^*)J_{22}(\textbf{P}^*)-J_{12}(\textbf{P}^*)J_{21}(\textbf{P}^*).
\end{align*}

In order to satisfy the trace-determinant criterion and guarantee that the eigenvalues of the characteristic equation have negative real parts, it is necessary for $\tau$ to be negative and for $\delta$ to be positive. Therefore, 
for $y<y_{max}$ to be true, we need $R_r(h_s +1)<R_s +h_s $.
Also, since $x= \dfrac{R_s - 1}{R_s}-\dfrac{R_s + h_s}{R_s}y$ in \eqref{2.14},  we have $y= \dfrac{R_s -1 - x R_s}{R_s + h_s}$. Hence we substitute these expressions at $(x^*,y^*)$ to obtain
\begin{align*}
    \tau= -\frac{\gamma h_{r2}}{h_s}R_sx^{*} -\gamma\left[h_{r1}q\frac{x^{*}}{y^{*}}+R_ry^{*}\right]=-\left(\frac{\gamma h_{r2}}{h_s}R_sx^{*}+ \gamma\left[h_{r1}q\frac{x^{*}}{y^{*}}+R_ry^{*}\right]\right)<0.
\end{align*}

Additionally, 
\begin{align*}
    \delta &= \left(-\dfrac{\gamma h_{r2}}{h_s}R_sx^{*}\right) \left(-\gamma\left[h_{r1}q\frac{x^{*}}{y^{*}}+R_ry^{*}\right]\right)-\left(-\dfrac{\gamma h_{r2}}{h_s}(R_s+h_s)x^{*}\right)  \left(\gamma[h_{r1}q+(h_{r2}-R_r)y^{*})] \right)\\
    &=\dfrac{\gamma^2 h_{r2}}{h_s}x^{*}\left(R_s\left[h_{r1}q\frac{x^{*}}{y^{*}}+R_ry^{*}\right]+(R_s+h_s)[h_{r1}q+(h_{r2}-R_r)y^{*})]\right)
    \\
    &=\dfrac{\gamma^2 h_{r2}}{h_s}x^{*}\left(R_sh_{r1}q\frac{x^{*}}{y^{*}}+(R_s+h_s)h_{r1}q+h_s(h_{r2}-R_r)y^{*}+R_s h_{r2}y^{*}\right)
    \\
    &=\dfrac{\gamma^2 h_{r2}}{h_s}x^{*}\left(R_sh_{r1}q\frac{x^{*}}{y^{*}}+(R_s+h_s)h_{r1}q+a_{2}y^{*}\right)>0.
\end{align*}
This completes the proof. 
\end{enumerate}
\end{proof}

The ($R_s,R_r$) plane is subdivided based on the local stability conditions provided by Theorem~\ref{teo_stability}, as depicted in Figure \ref{regiones}. Five distinct open regions with different qualitative dynamics are identified and further elaborated in Table \ref{tableregiones}, which summarizes the results presented in Theorem~\ref{teo_stability}.

\begin{figure}[h]
    \centering
   \includegraphics[width=8.1 cm, height=6cm]{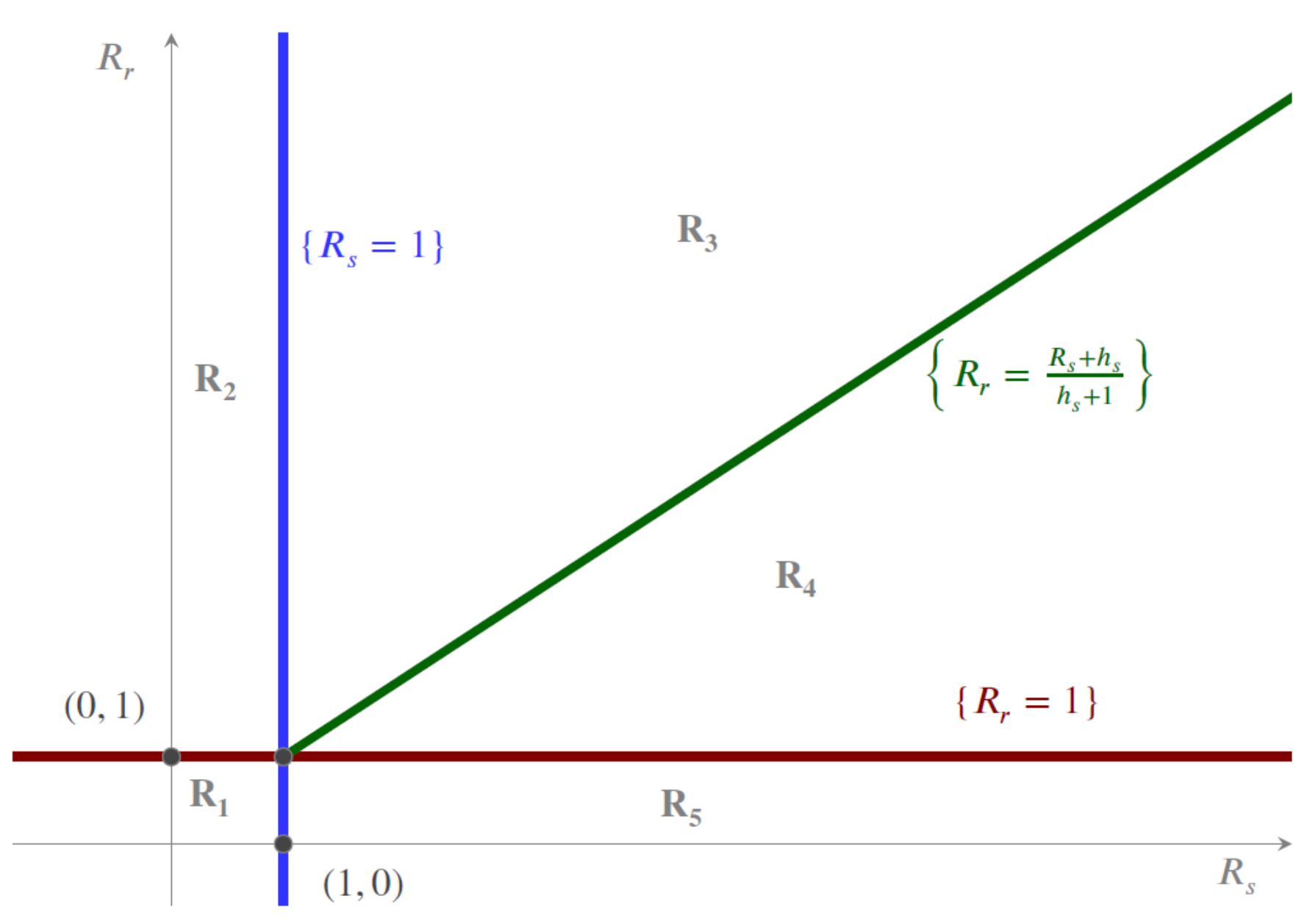}
    \caption{The stability regions of the equilibrium points $\textbf{P}_0$, $\textbf{P}_1$, and $\textbf{P}^*$ are determined in accordance with Theorem~\ref{teo_stability}.}
    \label{regiones}
\end{figure}


\begin{table}[h]
\centering
	\caption{
The stability regions for the equilibrium points of System \eqref{model_res2} are determined following Theorem~\ref{teo_stability} and Figure~\ref{regiones}. In this context, \textit{LAS} stands for \textit{locally asymptotically stable}.}
\begin{tabular*}{\tblwidth}{@{}LLLLLL@{}}
\toprule
Equilibrium    & $\mathbf{R_1}$ &  $\mathbf{R_2}$ & $\mathbf{R_3}$ & $\mathbf{R_4}$ & $\mathbf{R_5}$   \\
\midrule
$\mathbf{P_0}$  &       \textit{LAS}  & Unstable & Unstable &  Unstable & Unstable               \\
$\mathbf{P_1}$  &     Does not exist &   \textit{LAS} &  \textit{LAS} & Unstable   & Does not exist                         \\
$\mathbf{P^*}$ &        Does not exist & Does not exist & Does not exist &  \textit{LAS} &  \textit{LAS}                   \\
                    \bottomrule
	\end{tabular*}
 \label{tableregiones}
\end{table} 

To visually  the stability and existence conditions of the equilibrium points within each region represented in Figure \ref{regiones}, we employ the parameter values specified in Table \ref{tablefixed} (refer to Section \ref{sec:case}). Specifically, we present the phase portraits of System \eqref{model_res2} for each equilibrium point in Figure \ref{Phase}.

\begin{figure}[h]
    \centering
    \subfigure[ $ \mathbf{P}_{0}$ is \textit{LAS} ]{
    \includegraphics[scale=0.23]{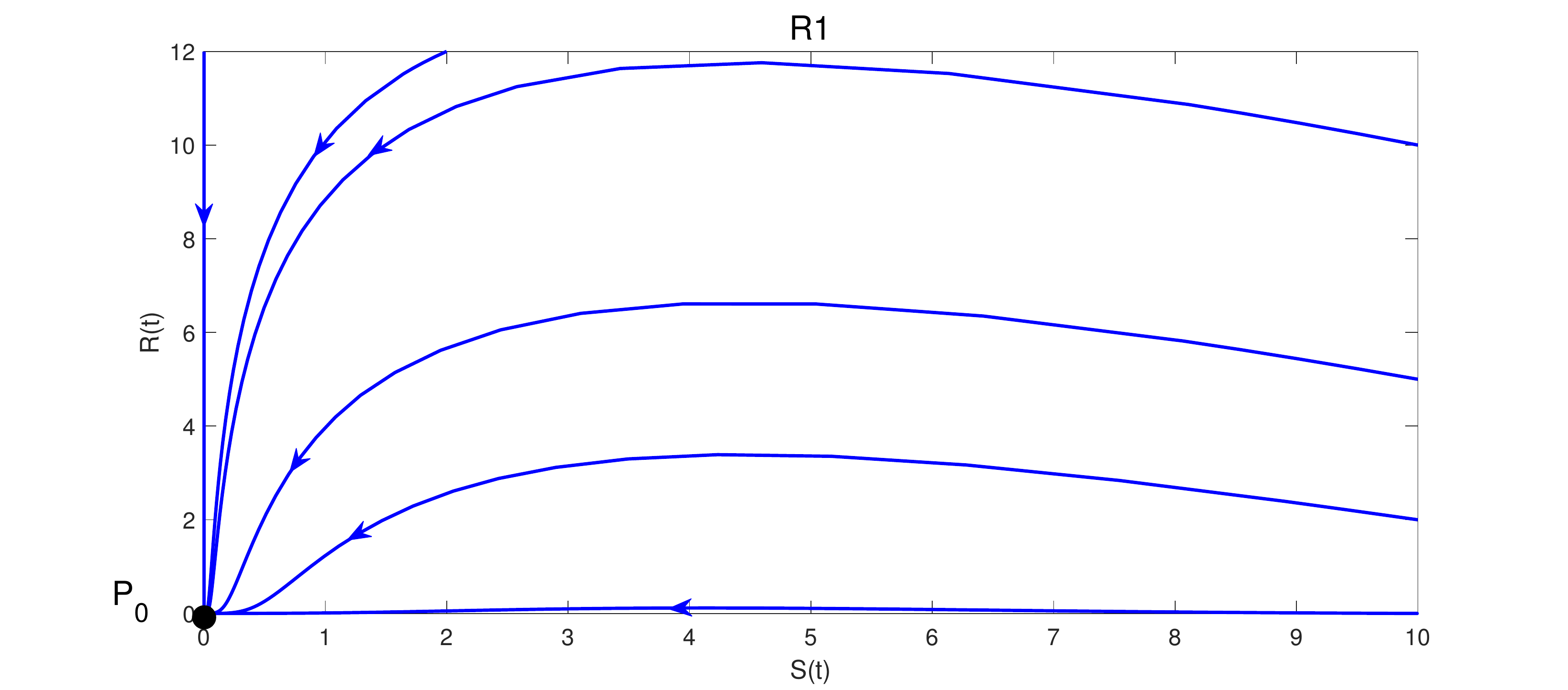}}
    \subfigure[$\mathbf{P_{1}}$ is \textit{LAS} and $\mathbf{P_{0}}$ is a saddle point]{ \includegraphics[scale=0.23]{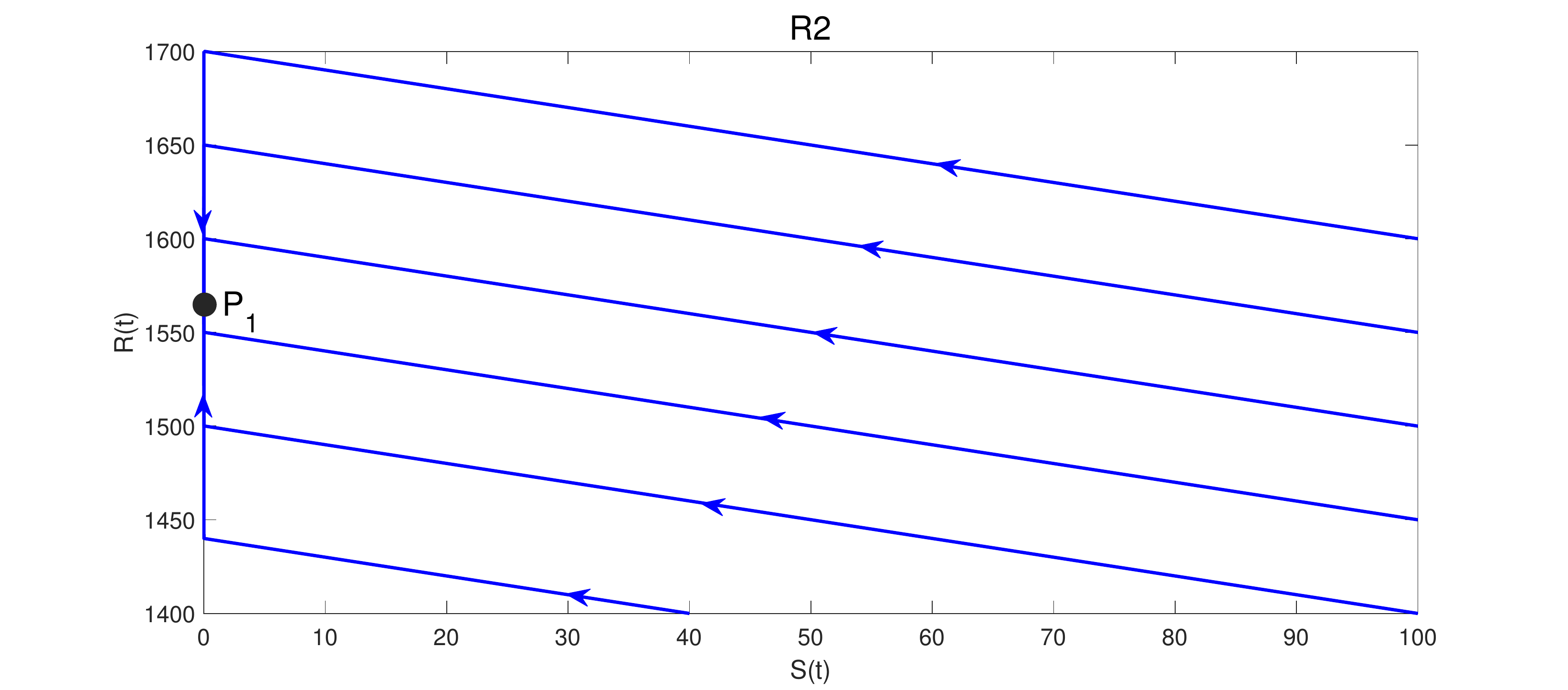}}
    \subfigure[$\mathbf{P_{1}}$ is \textit{LAS} and $\mathbf{P_{0}}$ is unstable]{\includegraphics[scale=0.23]{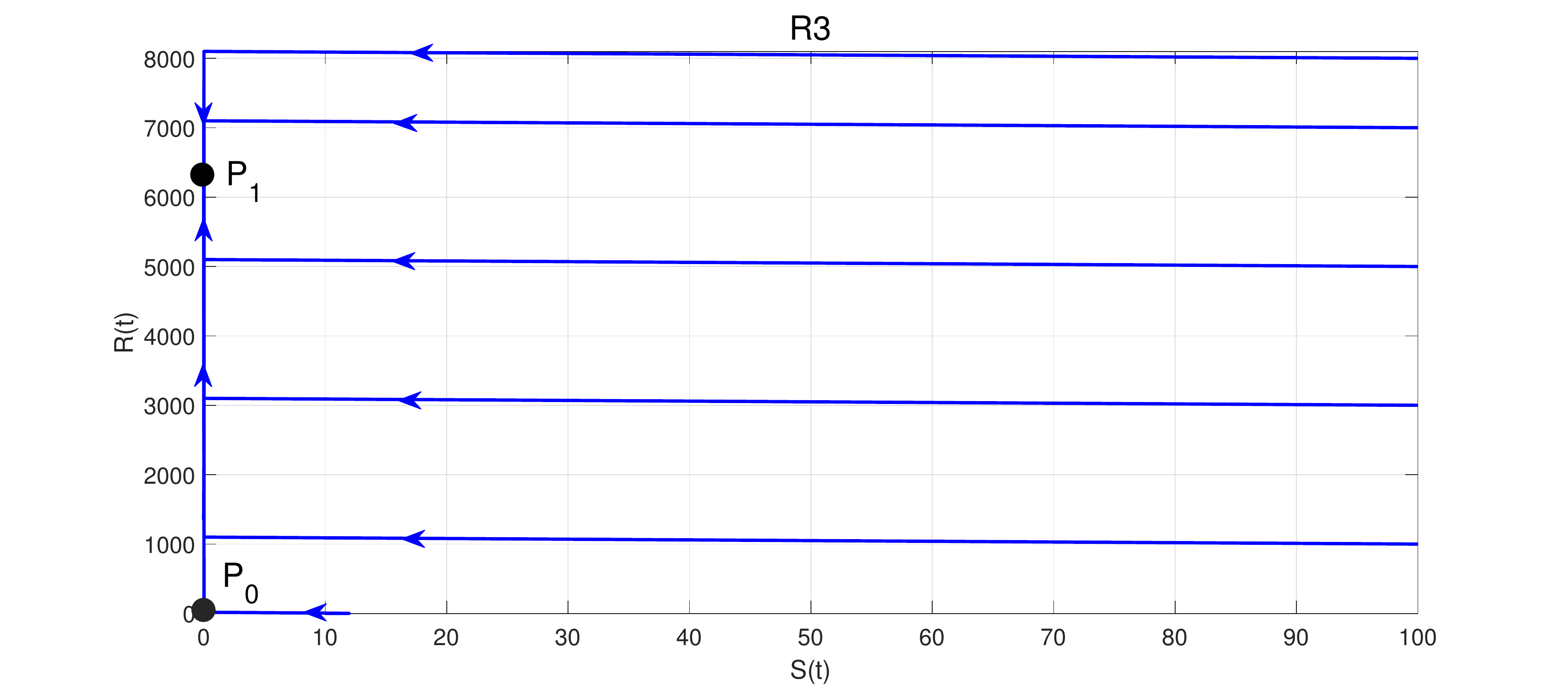}}
    \subfigure[ $\mathbf{P^{*}}$ is\textit{ LAS} and $\mathbf{P_1}$ and  $\mathbf{P_{0}}$ are unstable]{   \includegraphics[scale=0.2]{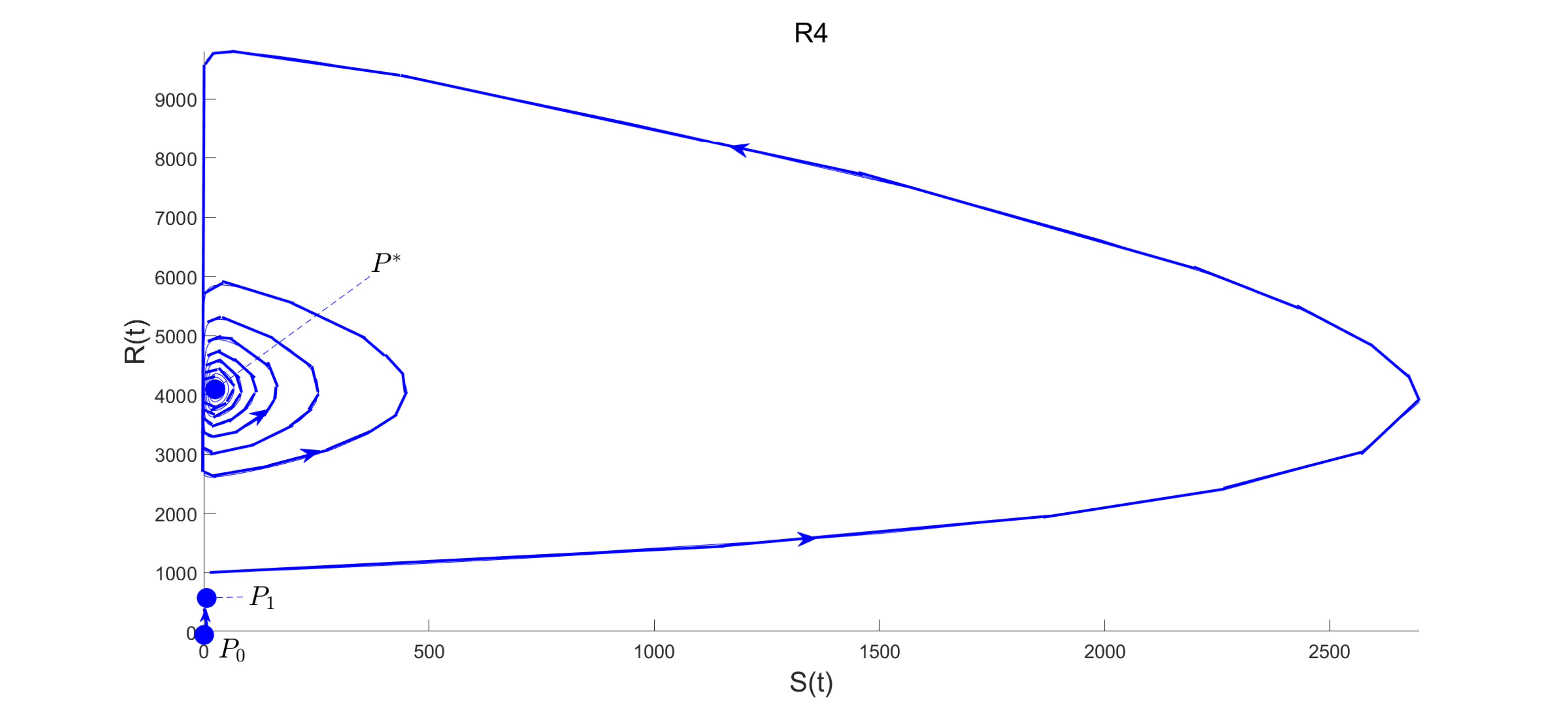}}
    \subfigure[$\mathbf{P_{*}}$ is \textit{LAS} and $\mathbf{P_{0}}$ is unstable]{\includegraphics[scale=0.2]{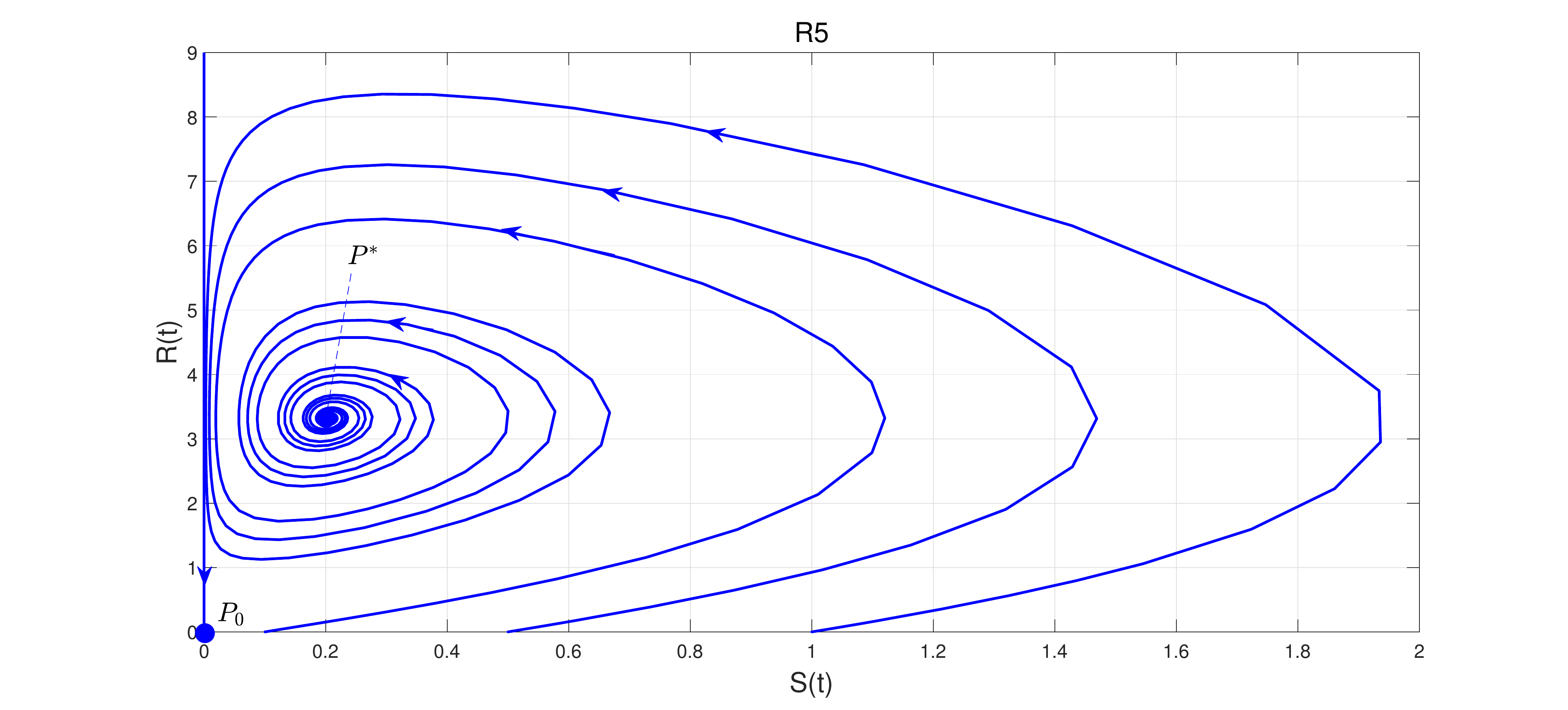}}
    \caption{Phase portraits of Model \eqref{model_res2}. (a) Dynamics in the  region $ \mathbf{R_{1}} $ where $ \mathbf{P}_{0}$ is \textit{LAS}. Here the threshold's values are:
$R_s=0.17$, $h_s=0.22$, $h_{r2}=8.33$, $R_r=0.53$, $h_{r1}=4166.66$.  The values of the mutation parameters and immune response  are $q=0.00098$, and $\gamma=0.00012$, respectively. (b) Dynamics in $\mathbf{R_{2}}$ where $\mathbf{P_{1}}$ is \textit{LAS} and $\mathbf{P_{0}}$ is a saddle point. Here,
$R_s=0.99$, $h_s=12.41$, $h_{r2}=158.73$, $R_r=1.01$, $h_{r1}=79365.07$. And $q=1.48e-5$  $\gamma=6.3e-6$. (c) Dynamics in  $\mathbf{R_{3}}$ where $\mathbf{P_{1}}$ is \textit{LAS} and $\mathbf{P_{0}}$  unstable. Here,
$R_s=1.01$, $h_s=0.12$, $h_{r2}=1.66$, $R_r=1.06$, $h_{r1}=150000$. And, $q=1.47e-5$,  $\gamma=6e-6$. D. Dynamics in   $\mathbf{R_{4}}$ where $\mathbf{P^{*}}$ is \textit{LAS}, and $\mathbf{P_1}$ and  $\mathbf{P_{0}}$ are unstable. Here,  $R_s=1.21$, $h_s=15.15$, $h_{r2}=157.38$, $R_r=1.007$, $h_{r1}=78690.58$. And  $q=1.47e-5$, $\gamma=6.35e-6$. (e). Dynamics in  $\mathbf{R_{5}}$ where $\mathbf{P_{*}}$ is \textit{LAS} and $\mathbf{P_{0}}$ is unstable.  Here, $R_s=1.05$, $h_s=1.31$, $h_{r2}=11.11$, $R_r=0.71$, $h_{r1}=5555.55$. And  $q=0.00014$,   $\gamma=9e-5$.} \label{Phase}
  \end{figure}
   
\section{The optimal control analysis} \label{sec:control}
 
This section discusses the optimal control problem formulated in Model \eqref{model}. We first analyze the set-up and  show the existence of the controls. Here, $h_1(t)$ and $h_2(t)$ are time-dependent. We assume an initial time $t_0=0$, a fixed final time $T>0$ ---which represents the implementation time of the control strategies--- and we assume that $\textbf{X}_1$ is a free dynamic variable at the end time. In contrast, the coordinates of the initial condition $\textbf{X}_0$ are the coordinates of a non-trivial equilibrium of Model~\eqref{model}.
\par
First, we check the existence and uniqueness of controls $h_1$ and $h_2$. To do this, using arguments similar to~\cite{PTU18}, we will verify that the following properties are fulfilled:  

\begin{enumerate}
	\item[(i)] The set of all solutions of \eqref{model} with the corresponding control functions in
$\mathcal{U}$ is not empty.
	\item [(ii)] The right-hand side of \eqref{model} is continuous,  bounded above by a bounded sum
control and state, and it can be written as a linear function of the control variables with coefficients
that depend on time and state variables.
\item[(iii)] The integrand in the cost function $L_{S,R}(h_1, h_2)=cR+(w_1+w_2h_1)h_1+(b_1+b_2h_2)h_2 $ is convex in $\mathcal{U}$ and also satisfies
\[cR+(w_1+w_2h_1)h_1+(b_1+b_2h_2)h_2  \geq k_1 |\textbf{h}|^\delta - k_2, \ \text{where} \ \; k_1 , k_2 > 0,\  \delta > 1.
	\]
\end{enumerate}
  
Assumption (i) is confirmed by Theorem 1 in \cite{Revista_inv}. By rewriting the right-hand side of Model \eqref{model} as $\dot{\textbf{X}} = f(\textbf{X}) + g(\textbf{X}) \textbf{h}$, Assumption (ii) follows immediately. Additionally, the integrand in the cost function is obviously convex because the Hessian matrix of $L_{S,R}(h_1, h_2)$ over the set of admissible controls is 
$$M=\begin{pmatrix}
	w_2& 0 \\
	0 & b_2
	\end{pmatrix}.$$
	
 Since 
$Sp(M)=\{w_2,b_2\} \subset \mathbb{R}_{+}^{*}$, then $L_{S,R}(h_1, h_2)$ is strictly convex in $\mathcal{U}$. By letting $k_2=a$ and $k_1=\min\{ w_2, b_2\}$, the integrand in the cost function can be written as
\begin{equation*}
cR+(w_1+w_2h_1)h_1+(b_1+b_2h_2)h_2 \geq -cR+w_2h_1^2+b_2h_2 2 \geq -k_2+k_1 |\textbf{h}|^2,
\end{equation*}
which satisfies Assumption (iii).
\par 

To determine the optimal controls for Problem \eqref{model}, we employ Pontryagin's principle for bounded controls, as described in \cite{PBGM62}. The corresponding Hamiltonian is then given by:

\begin{equation}\label{hamiltonian}
\begin{array}{rcl}
H&=&cR+(w_1+w_2h_1)h_1+(b_1+b_2h_2)h_2 \\ \\
 &&+\left[ \beta_SS\left(1-\dfrac{S+R}{K}\right)-(\bar\alpha\Lambda+\bar \gamma)S-(1-h_1(t))\bar q\Lambda S-(1-h_2(t))aRS  \right]\lambda_1 \\ \\
 & &+\left[\beta_RR\left(1-\dfrac{S+R}{K}  \right)+(1-h_1(t))\bar q\Lambda S +(1-h_2(t))aRS-\bar\gamma R  \right] \lambda_2,
\end{array}
\end{equation}

where the vector $\bm \lambda=(\lambda_1, \lambda_2)$ is the set of \textit{adjoint variables} that determine the adjoint system associated with Problem \eqref{model}. The optimal system is defined by the adjoint system and state equations. The following result can be established.

\begin{theorem}\label{thm:teocontrol}
	For Problem \eqref{model}, there exists a corresponding optimal solution $(S^*(t), R^*(t))$ that minimize $J(\textbf{h})$ in $[0, T]$. Moreover, there exits an adjoint function $\bm\lambda(t)=(\lambda_1(t), \lambda_2(t))$ that satisfies the following adjoint system:
\begin{equation}\label{adjoint_system}
	\left\{\begin{array}{ll}
	 \dot{\lambda}_1 =& -\beta_S\lambda_1+\dfrac{\beta_S}{K}(2S+R)\lambda_1+(\bar\alpha\Lambda+\bar\gamma)\lambda_1-(1-h_1)\bar q\Lambda(\lambda_2-\lambda_1) \\ \\
                     &-(1-h_2)aR(\lambda_2-\lambda_1)+\dfrac{\beta_RR}{K}\lambda_2,  \\ \\
	\dot{\lambda}_2 =&  -c-\beta_R\lambda_2+\dfrac{\beta_S S}{K}\lambda_1 +\dfrac{\beta_R}{K}(S+2R)\lambda_2-(1-h_2)aS(\lambda_2-\lambda_1)+\bar\gamma\lambda_2,
	\end{array}\right.
\end{equation}
\noindent
with transversality condition $\lambda_i(t)=0$ for $i=1,2$, such that the optimal controls satisfy:
\begin{align}\label{optimal_controls}
h_1^*&=\min \left\{\max\left\{0, \frac{-w_1-\bar q\Lambda S(\lambda_1-\lambda_2)}{2w_2}  \right\} , 1 \right\}, \\ \nonumber
h_2^*&=\min \left\{\max\left\{0, \frac{-b_1-aRS(\lambda_1-\lambda_2)}{2b_2}  \right\} , 1 \right\}.
\end{align}
\end{theorem}

The details of the proof can be located in Appendix \ref{appendix:A3}.

\section{Numerical simulations}\label{sec:case}
This section presents a case study that examines the population growth of \textit{E. coli} using Model \eqref{model}. \textit{E. coli} is a bacterium that typically inhabits the intestines of warm-blooded organisms. Fresh fecal matter can grow quickly but gradually decreases over time \cite{berry2005cattle}. Fecal-oral transmission is the primary route by which pathogenic strains of \textit{E. coli} cause disease. Various antimicrobials have been used to combat \textit{E. coli}, each with varying levels of effectiveness. For instance, Ampicillin, which is widely prescribed in Europe and Latin America \cite{mortazavi2019pattern}. The pathogen is highly resistant to ampicillin. Amikocin, gentamicin, or tobramycin were found to be the most effective drugs \cite{rahal1976bactericidal}. This information is essential to the case study, as it guides the selection of parameters for the mathematical model.
\par 
According to \cite{riano2022contribution}, the percentage of resistance of {\it E coli} and percentage of MDR in {\it E coli} vary in European countries, based on their regions. These regions can be categorized into three clusters: Southern European countries, Central European countries, and Northern European countries. In colder countries, the rate of elimination of susceptible bacteria by antibiotics is high, whereas antibiotics are less effective in warmer and more southerly countries \cite{pepi2021antibiotic}. The varying rates of AMR between cold and warm countries can be attributed to several factors 
related to environmental conditions, healthcare practices, and microbial behaviour. Firstly, cold countries often experience prolonged winters with low temperatures. This cold environment can create favourable conditions for the survival and persistence of specific bacterial strains \cite{hance2007impact}. Secondly, cold countries typically have robust healthcare infrastructures, including well-established infection control measures and stringent antimicrobial prescribing guidelines. In contrast, some warm countries may need help implementing effective infection control practices, ensuring proper antimicrobial stewardship, and maintaining adequate healthcare infrastructure \cite{oaks1992emerging}. Thirdly, in warm countries, agricultural activities like livestock farming may involve using antimicrobial as growth promoters or preventive measures. This agricultural use of antimicrobial can contribute to selecting and disseminating resistant pathogen in the environment, food chain, and communities \cite{oaks1992emerging}. Finally, microbes can adapt and evolve differently in cold and warm environments. Cold-adapted microbes may possess unique genetic traits that enable them to thrive in cold conditions and resist antimicrobial agents \cite{wani2022microbial}. 
\par 
We utilized the data from \cite{european2010european} to calculate the average percentages of resistance and MDR of {\it E. coli} to  gentamicin and amoxicillin in \textit{E. coli}, as part of our numerical experiments. According to those data,  amoxicillin (from the aminopenicillin family) is considered a less effective antibiotic for treating {\it E. coli} infections compared to gentamicin (from the aminoglycosides family) in all regions of Europe \cite{darras1996synergy}. In fact, some studies have provided evidence supporting the effectiveness of gentamicin over amoxicillin in the treatment of infections caused by \textit{E. coli} \cite{olson2002biofilm, olorunmola2013antibiotic}.
\par 
Using data found in \cite{european2010european} and \cite{riano2022contribution}, certain parameter values were assumed constant throughout the experiment, including the fixed control parameters which operated at 50\% efficiency ($h_1$=$h_2$=0.5). These specific values can be found in Table \ref{tablefixed}. On the other hand, the remaining parameters such as the elimination rate of sensitive bacteria ($\bar \alpha$), the mutation rate ($\bar q$), the administration rate of antibiotics ($\Lambda$), and the elimination rate of antimicrobial by immune system ($\bar \gamma$) were considered as variables, dependent on both the geographical zone (North, Center, or South of Europe) and the type of antibiotic used (amoxicillin or gentamicin). Detailed information regarding these parameter values can be found in Tables \ref{table-region} (for $\bar \alpha$ and $\bar q$ variations), Table \ref{table-Lambda} (for $\Lambda$ variation), and Table \ref{table-gamma} (for $\bar \gamma$ variation).
\par 
Consequently, our numerical experiments aimed to illustrate three distinct scenarios. The first scenario involved examining the impact of varying $\bar \alpha$ and $\bar q$ for both amoxicillin and gentamicin across different geographical zones in Europe. The second and third scenario focused on to analyze in depth the most critical scenario (supply of amoxicillin in southern European countries).   Particularly, the second scenario to explore the effects of varying $\Lambda$, and the third scenario involved investigating the influence of varying $\bar \gamma$.
Towards the conclusion of this section, we conduct numerical experiments involving variable controls ($h_1$ and $h_2$).

\begin{table}[h]
\centering
	\caption{Fixed parameter values of Model (\ref{model}). Time in hours.}
\begin{tabular*}{\tblwidth}{@{}LLL@{}}
\toprule
Parameter   & Value &Reference    \\
\midrule
$\beta_S$  &    8   &\cite{european2010european}, \cite{riano2022contribution}                \\
$\beta_R$  &   0.64   &\cite{european2010european}, \cite{riano2022contribution}                          \\
$h_1$ &  0.5 &Assumed  \\
$h_2$ &  0.5 &Assumed\\
$a$        &  1     &Assumed          \\
$K$        & 1e5                         &\cite{european2010european}, \cite{riano2022contribution}  \\
		\bottomrule
	\end{tabular*}
 \label{tablefixed}
\end{table}


\begin{table}[h]
\centering
	\caption{Values of the elimination of sensitive bacteria and the mutation rates for contact with amoxicillin and gentamicin according to  the geographical zone of Europe. Time is given in hours. These data were taken from \cite{european2010european}.}
	\begin{tabular*}{\tblwidth}{@{}LLLLLLLL@{}}
			\toprule	
Parameter & \multicolumn{2}{c}{North Europe} & \multicolumn{2}{c}{Center Europe} & \multicolumn{2}{c}{South Europe} \\\hline
          & Amoxicillin  & Gentamicin  & Amoxicillin   & Gentamicin  & Amoxicillin  & Gentamicin  \\ \hline
$\bar \alpha$     & 0.56        & 0.92        & 0.4          & 0.88        & 0.36        & 0.79        \\
$\bar q$         & 0.44        & 0.088        & 0.60         & 0.12        & 0.64        & 0.21     \\
	\bottomrule
	\end{tabular*}
 \label{table-region}
\end{table}

\begin{table}[h]
\centering
	\caption{Values of the administration rate of gentamicin in the North of Europe. Time is given in hours. These data were taken from \cite{european2010european}.}
	\begin{tabular*}{\tblwidth}{@{}LLLLL@{}}
			\toprule	
Parameter   &Low & Standard& High\\
\midrule   
$\Lambda$    & 12  & 8 &4                   \\
	\bottomrule
	\end{tabular*}
 \label{table-Lambda}
\end{table}

\begin{table}[h]
\centering
	\caption{
	Values of the elimination rate of amoxicillin by the immune system in the South of Europe. RISH represents  robust immune system host; CISH represents compromised immune system host; and  SCISH represents  immune system significantly compromised host. Time is given in hours. These data were taken from \cite{european2010european}. }
	\begin{tabular*}{\tblwidth}{@{}LLLLL@{}}
			\toprule	
Parameter  & RISH & CISH & SCISH\\
\midrule    
$\bar\gamma$  & 2.4 & 1.5 & 0.9                     \\
	\bottomrule
	\end{tabular*}
 \label{table-gamma}
\end{table}


Figure \ref{fig:var_alpha_q} was generated through numerical simulations, following the administration of gentamicin and amoxicillin with parameters $\Lambda=8$ and $\bar \gamma=2.4$, using an initial condition of (1,0). The figure depicts the results of these simulations for sensitive and resistant bacteria, considering different values of the sensitive bacteria elimination rate ($\bar \alpha$) and the mutation rate ($\bar q$) in the three considered geographical regions of Europe. Here, it becomes apparent that the administration of amoxicillin results in higher densities of both bacterial populations compared to gentamicin. Furthermore, a noticeable disparity emerges among the three regions of Europe, with southern and central countries being more heavily impacted in terms of the prevalence of resistant bacteria, in contrast to northern countries.

\begin{figure}[h]
    \centering
  \subfigure[Sensitive bacteria treated with gentamicin] {\includegraphics[scale=0.2]{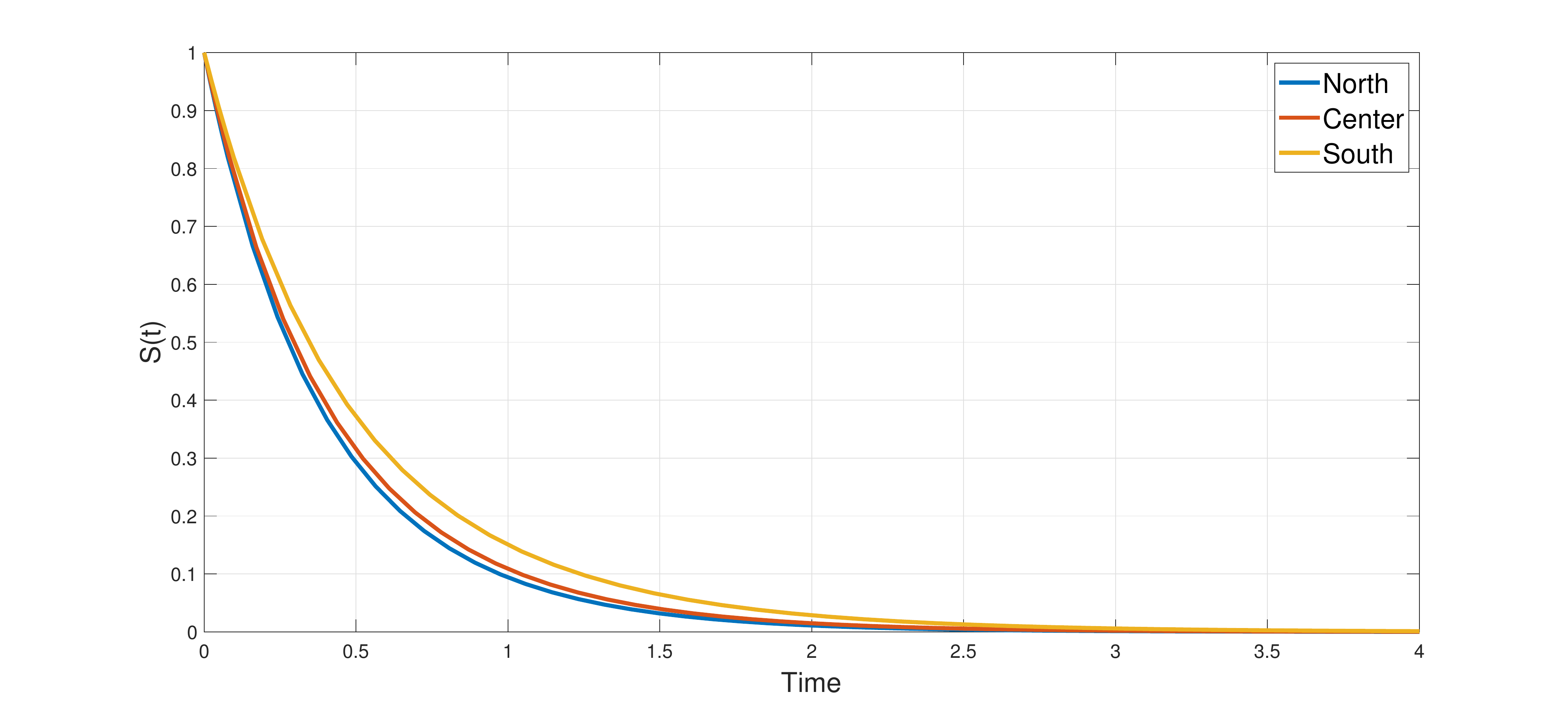}}
  \subfigure[Resistant bacteria treated with gentamicin] {\includegraphics[scale=0.2]{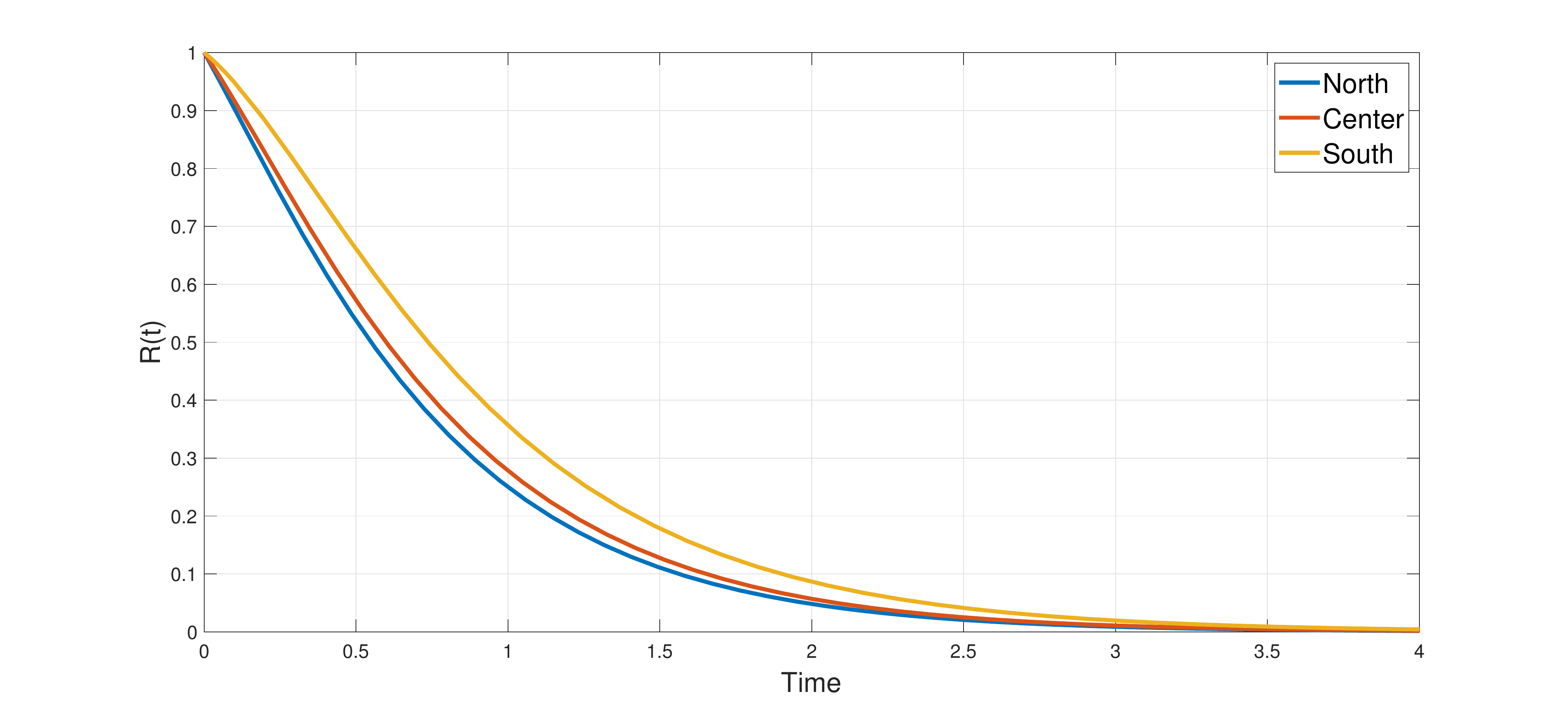}}
  \subfigure[Sensitive bacteria treated with amoxicillin] {\includegraphics[scale=0.2]{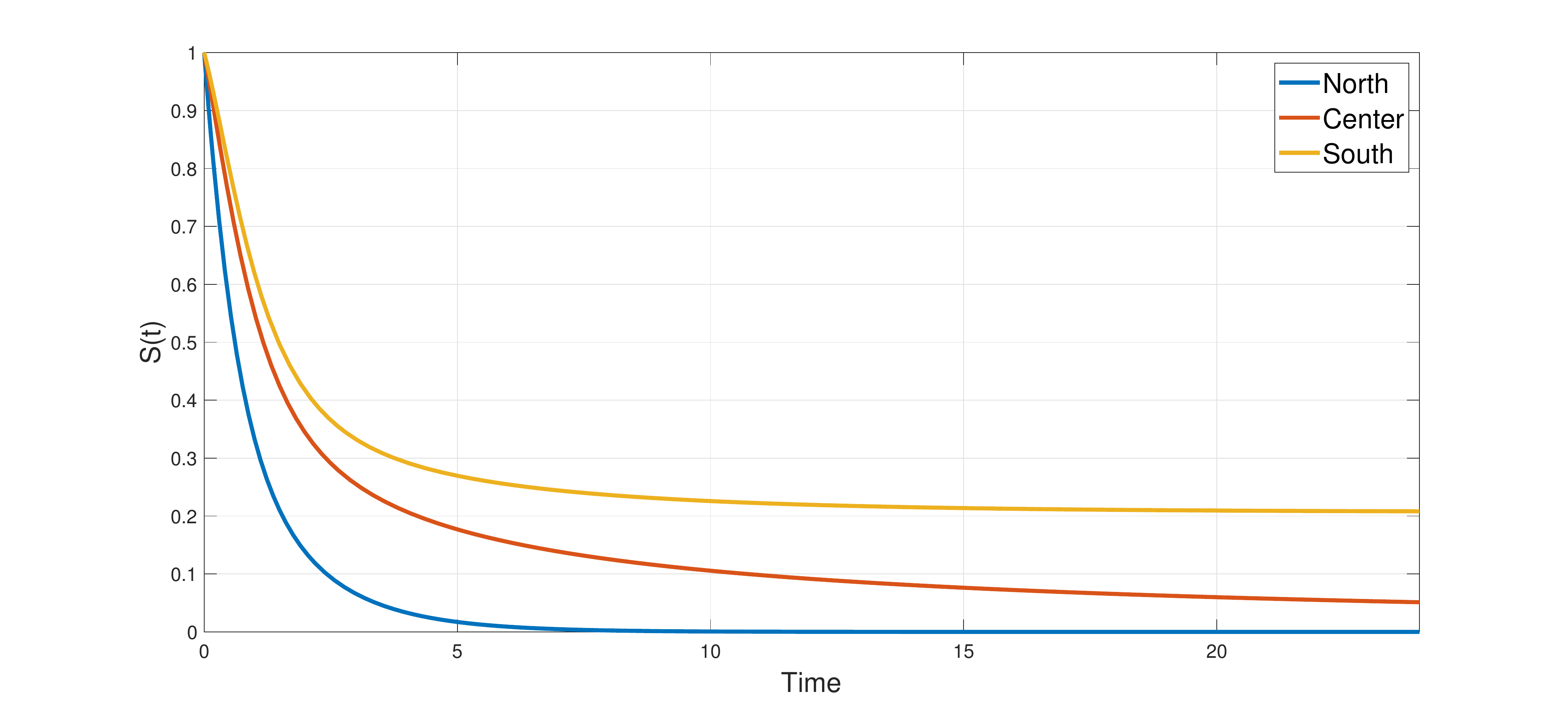}}
  \subfigure[Resistant bacteria treated with amoxicillin] {\includegraphics[scale=0.2]{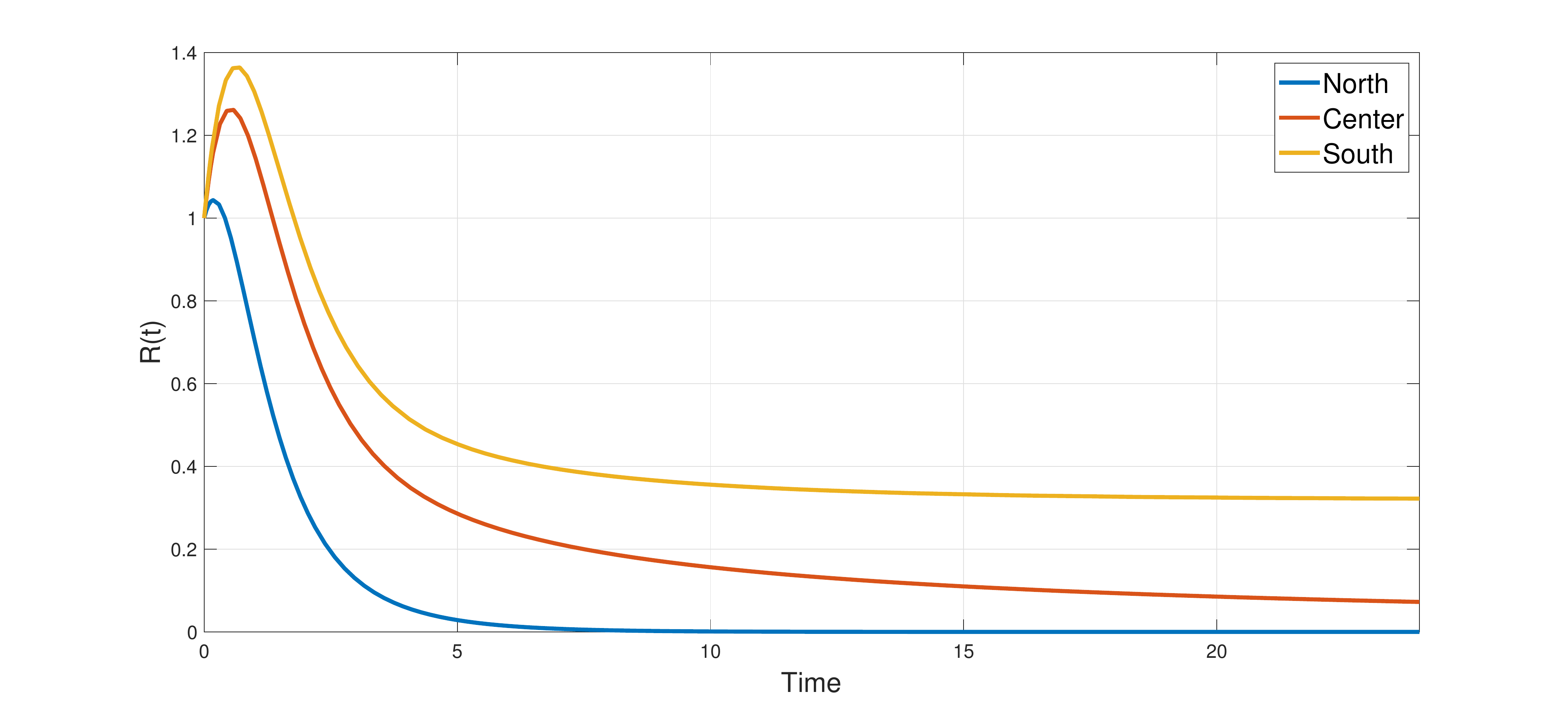}}
    \caption{Numerical simulations conducted for sensitive and resistant bacteria  after administering gentamicin (top) and amoxicillin (bottom) for different regions of Europe. Here, $\Lambda=8$ and $\bar \gamma=2.4$, and the initial condition is (1,0). }
    \label{fig:var_alpha_q}
\end{figure}

Figure \ref{fig:var_lambda} presents the bacterial densities of sensitive and resistant bacteria in Southern European countries when amoxicillin is administered. In this figure, we illustrate this scenario varying antibiotic supply rates $\Lambda$ (high, standard, and low), with fixed values of $\bar \alpha=0.92$, $\bar \gamma=2.4$, and $\bar q=0.44$. This figure shows  that as the frequency of antibiotic administration increases (lower $\Lambda$ values), the densities of both sensitive and resistant bacteria rise, with subtle oscillations observed in their behaviour. Conversely, reducing the frequency of antibiotic administration leads to a decrease in the populations of both type of bacteria. 

\begin{figure}[h]
   \subfigure[Sensitive bacteria.] {\includegraphics[scale=0.2]{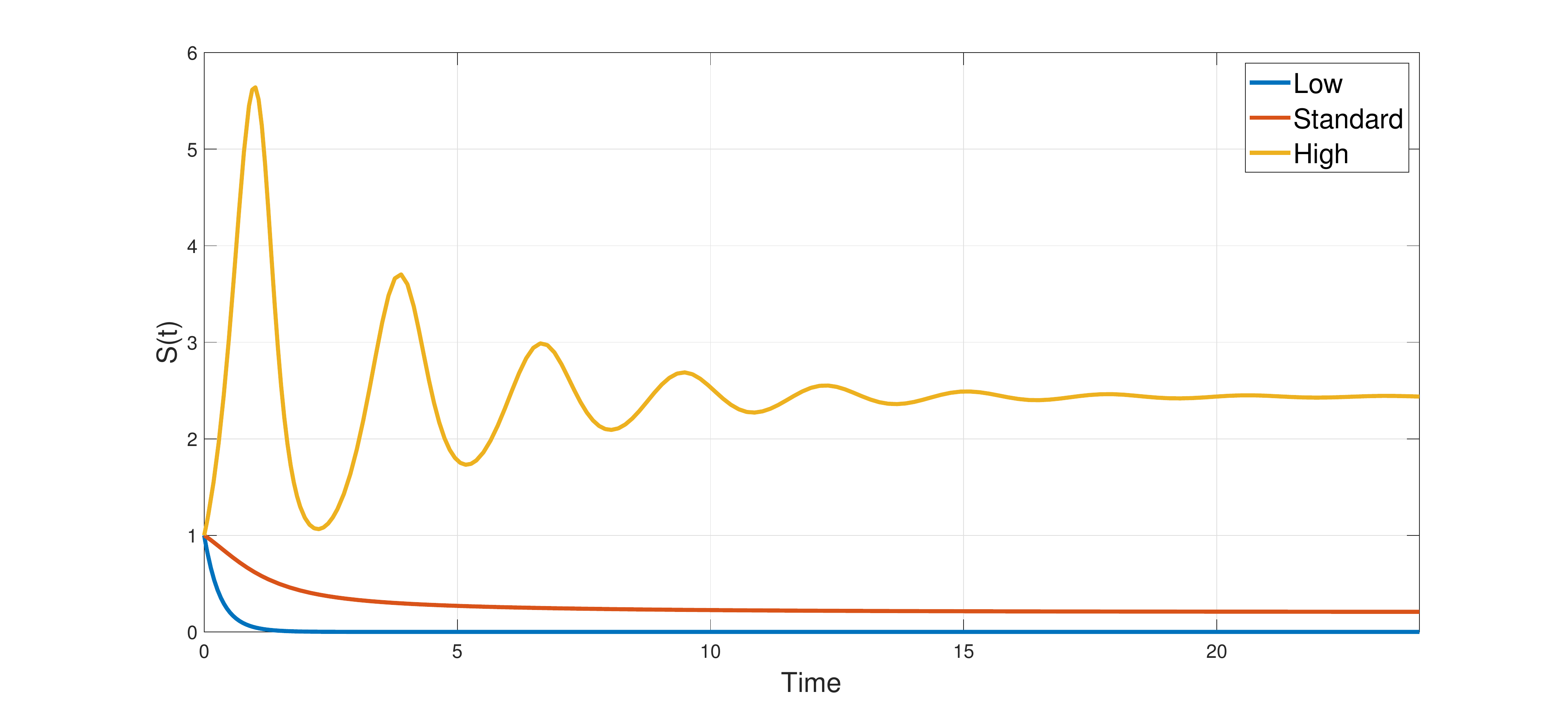}}
  \subfigure[Resistant bacteria.] {\includegraphics[scale=0.2]{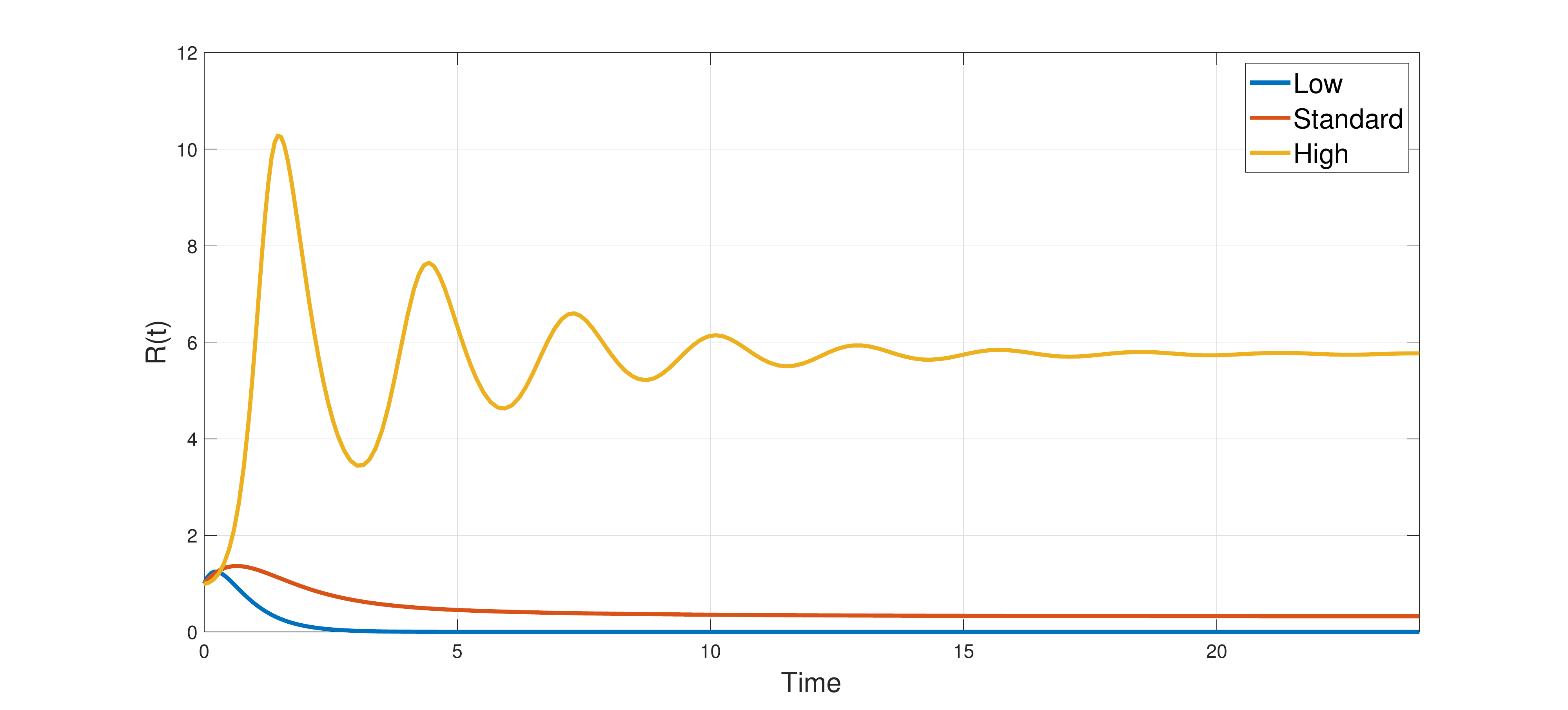}}
    \caption{Numerical simulations conducted for sensitive (a) and resistant bacteria (b) after administering amoxicillin in Southern Europe for different values of the antibiotic supply $\Lambda$. Here, $\bar \alpha=0.92$, $\bar \gamma=2.4$, and $\bar q=0.44$ and the initial condition is (1,0).}
    \label{fig:var_lambda}
\end{figure}


In Figure \ref{fig:var-gamma}, we investigate the influence of the elimination rate of amoxicillin by the host immune system, denoted as $\bar{\gamma}$,  considering an initial condition of (0,1). The study focused on three categories of hosts: those with a robust immune system (RISH), those with a compromised immune system (CISH), and those with a significantly compromised immune system (SCISH). The figure illustrates that when the host immune system is significantly compromised, sensitive bacteria exhibit a slight oscillatory behaviour. Furthermore, patients with a compromised immune system contribute the most to the burden of sensitive bacteria over time. Conversely, the dynamics for resistant bacteria differ, as patients with a significantly compromised immune system exhibit the highest burden of resistant bacteria.

\begin{figure}[h]
   \subfigure[Sensitive bacteria.] {\includegraphics[scale=0.2]{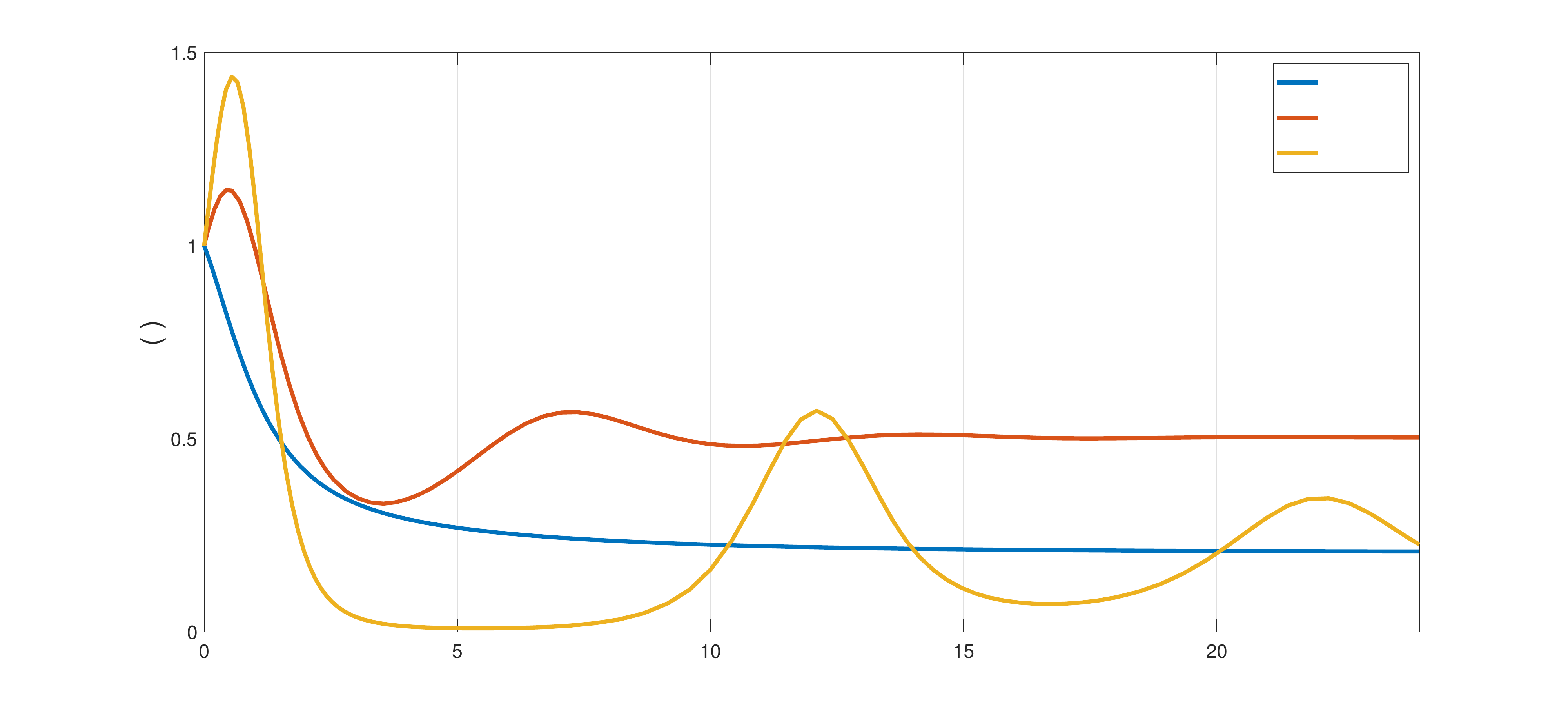}}
  \subfigure[Resistant bacteria.] {\includegraphics[scale=0.2]{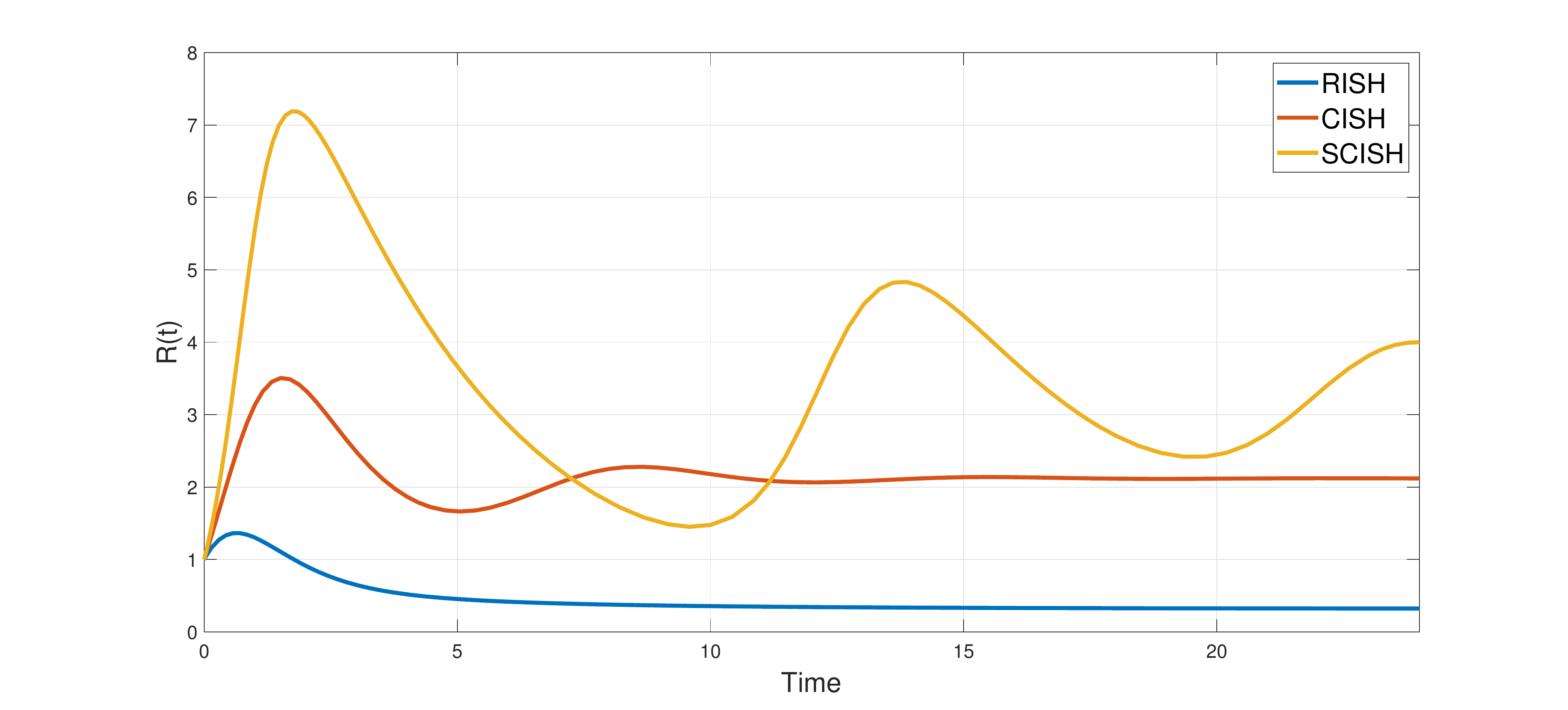}}
    \caption{Numerical simulations conducted for sensitive (a) and resistant bacteria (b) after administering amoxicillin in southern Europe for different immune response hosts $\bar \gamma$. Here,  $\bar \alpha$=0.56 and $\bar q$=0.44, and the initial condition is (1,0).}
    \label{fig:var-gamma}
\end{figure}


To conclude our numerical experiments, we incorporated  $h_1$ and $h_2$ as variables as in the optimal control problem defined in equation \eqref{model}. Our focus was to assess the efficacy of implementing controls in the Southern European countries when treating bacteria with gentamicin and amoxicillin (see Figure \ref{fig:control-genta-am}), thereby comparing the effectiveness of both controls for these antibiotics. 
Figure \ref{fig:control-genta-am} illustrates the results of the numerical experiments.  Here, the implementation of two control strategies was observed to effectively reduce both populations of bacteria, namely sensitive and resistant. Specifically, the control of sensitive bacteria resulted in rapid and slightly oscillatory dynamics, leading to their near elimination. Notably, the control variable $h_2$, responsible for managing mutations through HGT, exerted the greatest effort in controlling the sensitive bacterial population. In contrast, controlling resistant bacteria proved to be more challenging with both antibiotics, although their population decreased over time. Intriguingly, the control variable $h_1$, representing the control of mutations acquired through bacterial exposure to antibiotics, exhibited the highest effort in controlling the resistant bacterial population. These findings highlight the differential contributions of the control variables $h_1$ and $h_2$ in managing sensitive and resistant bacterial populations, respectively.
 
\begin{figure}[h]
\centering
\subfigure[Sensitive bacteria controlled with gentamicin]{\includegraphics[scale=0.2]{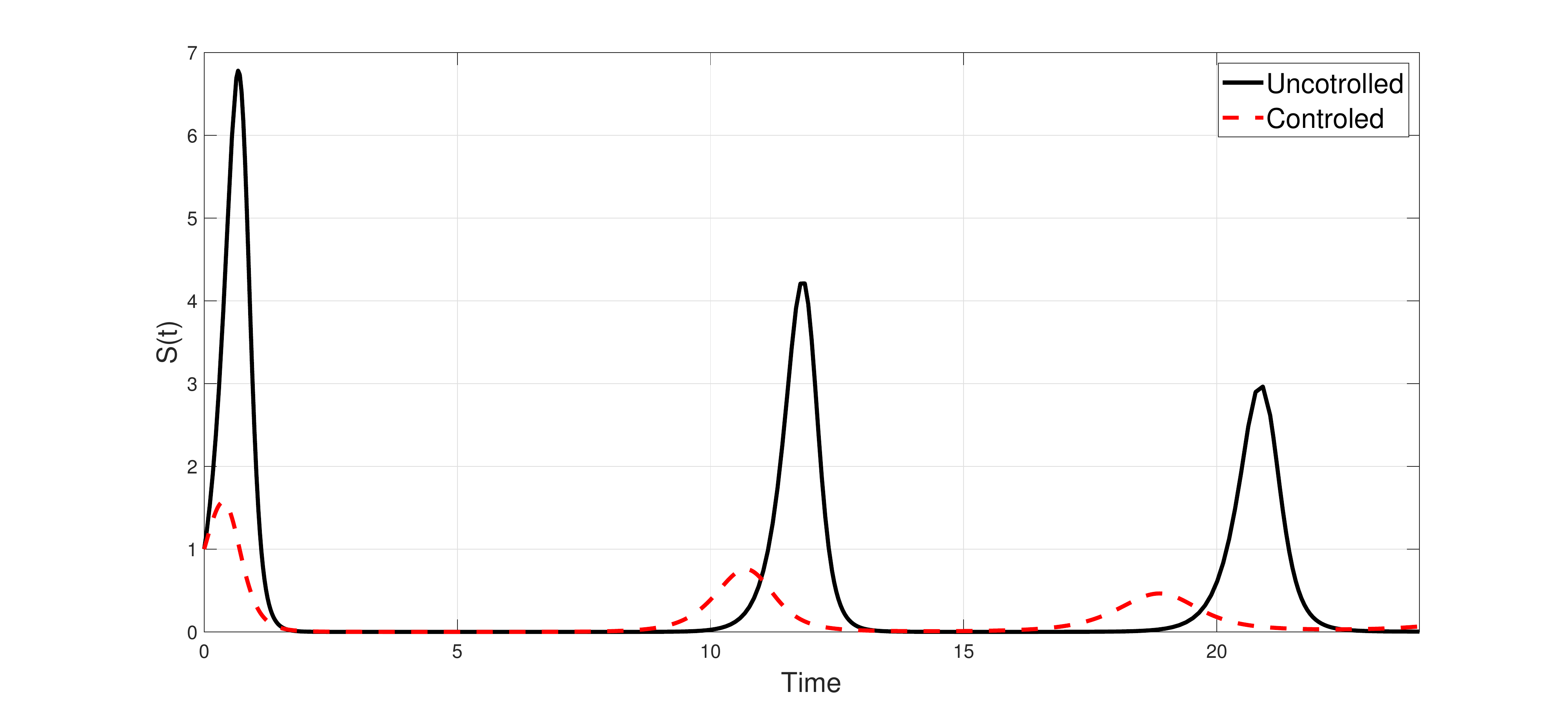}}
\subfigure[Resistant bacteria controlled with gentamicin]{\includegraphics[scale=0.2]{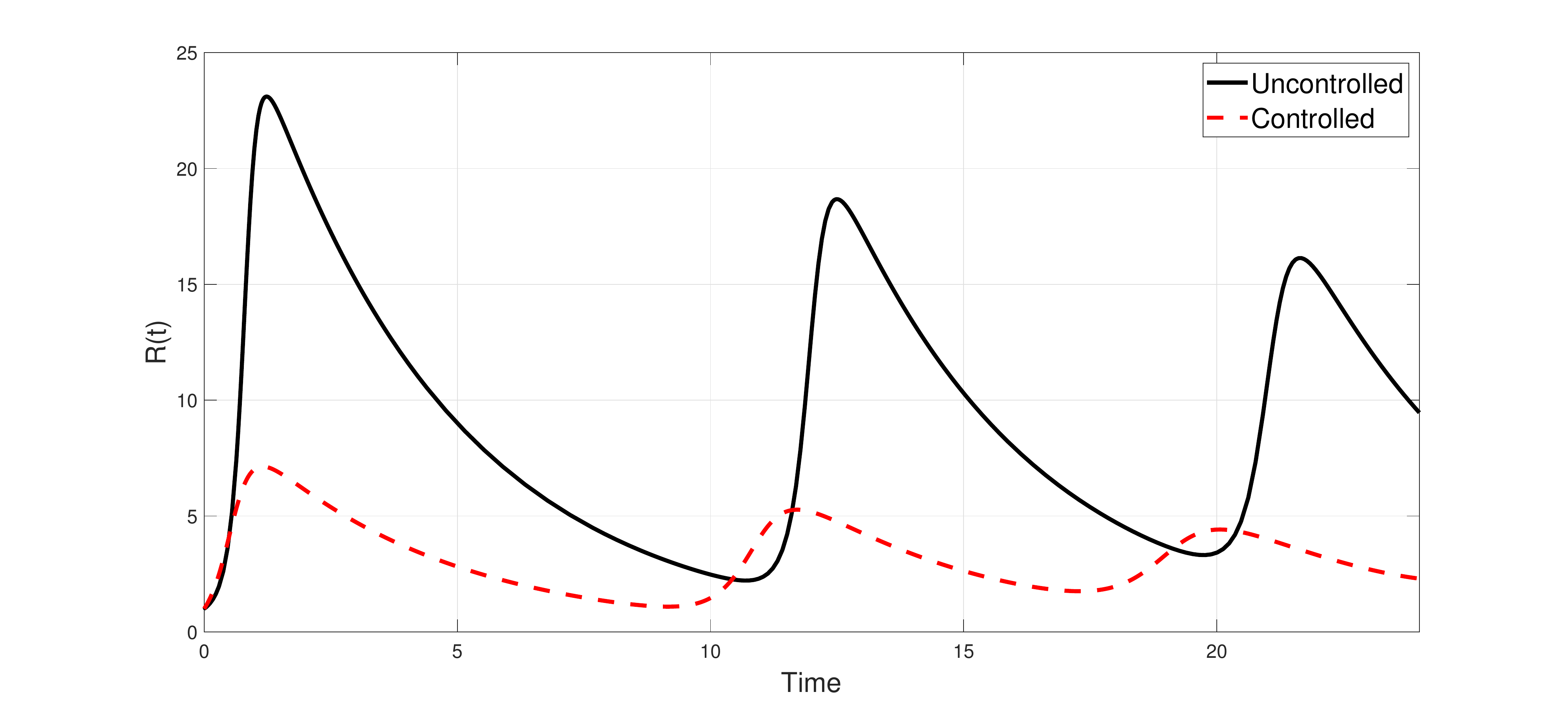}}
\subfigure[Sensitive bacteria controlled with amoxicillin]{\includegraphics[scale=0.23]{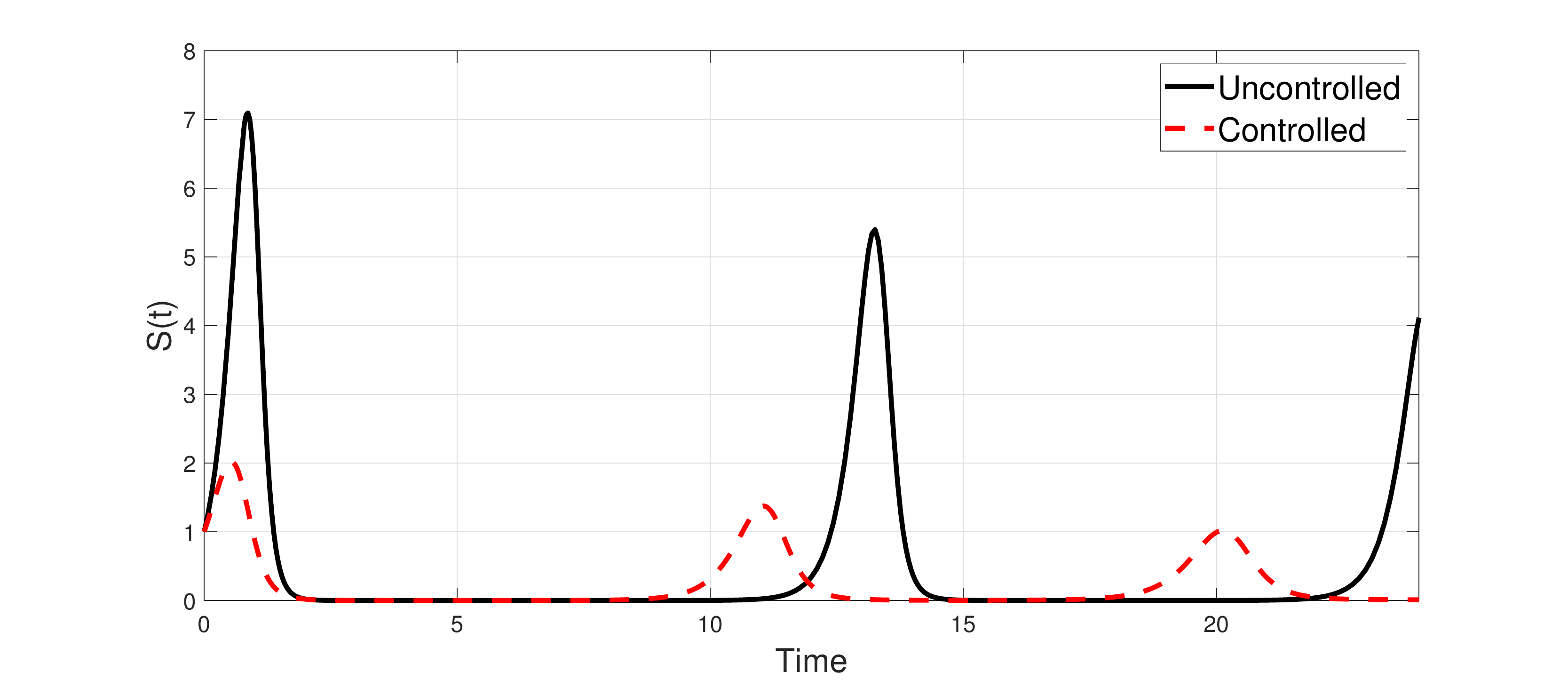}}
\subfigure[Resistant bacteria controlled with amoxicillin]{\includegraphics[scale=0.2]{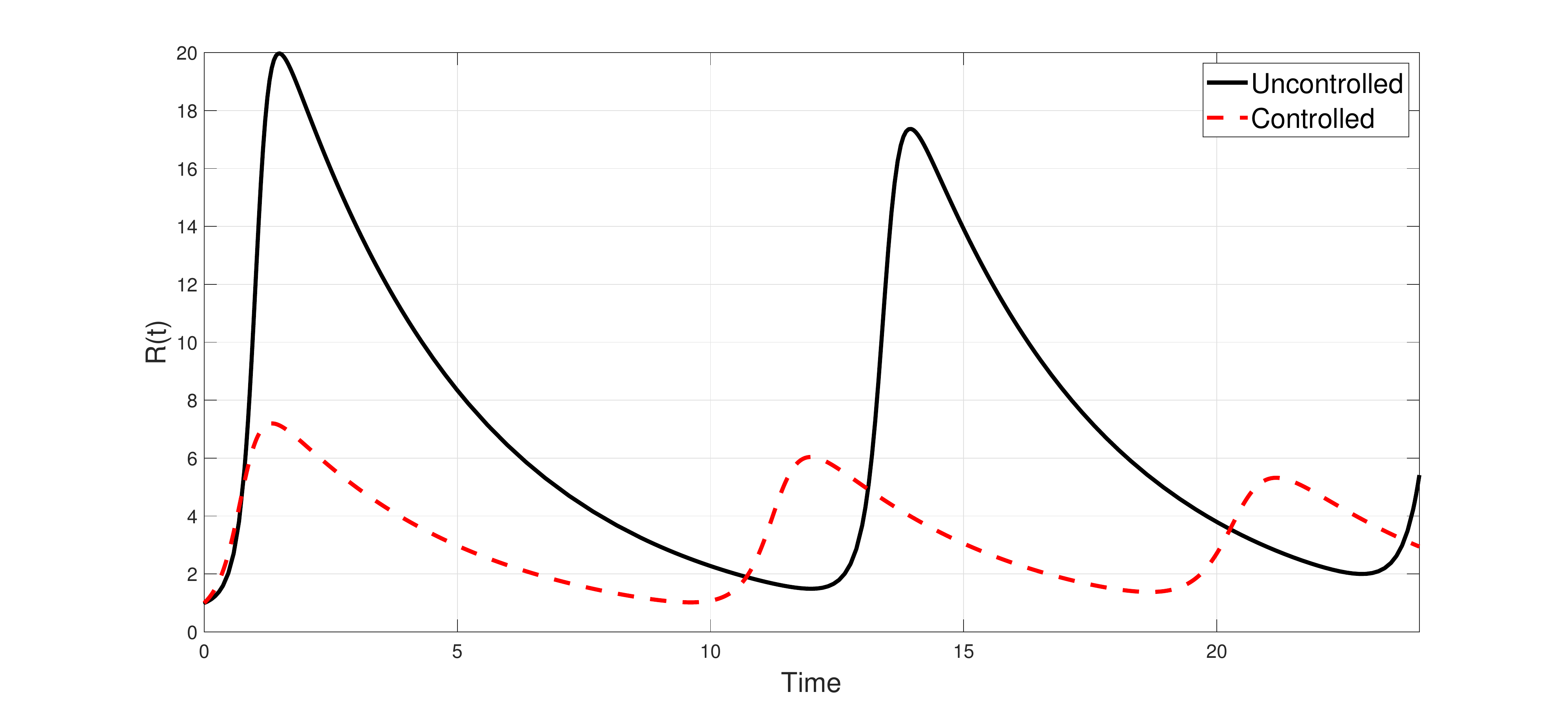}}
\subfigure[Controls behaviour with gentamicin]{\includegraphics[scale=0.2]{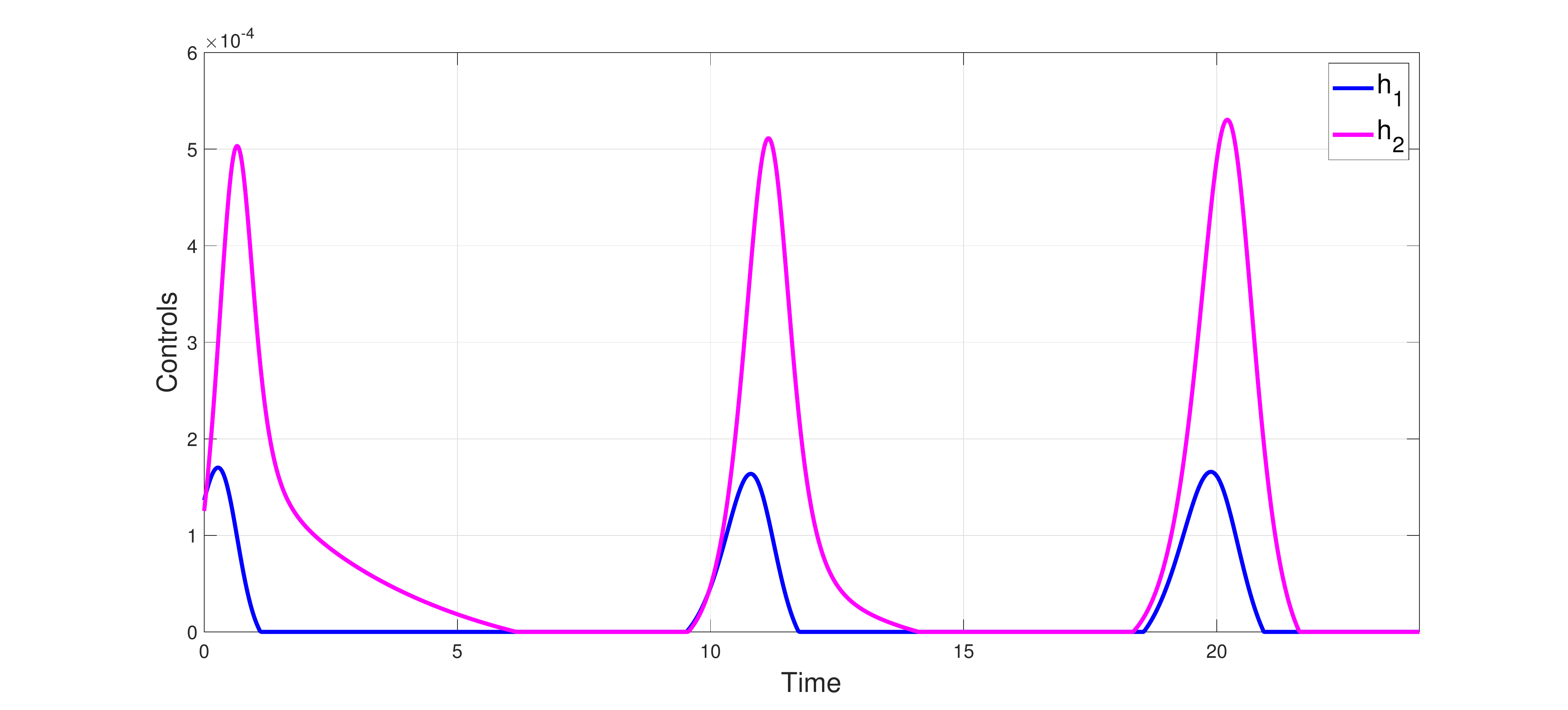}}
\subfigure[Controls behaviour with Amoxicillin]{\includegraphics[scale=0.2]{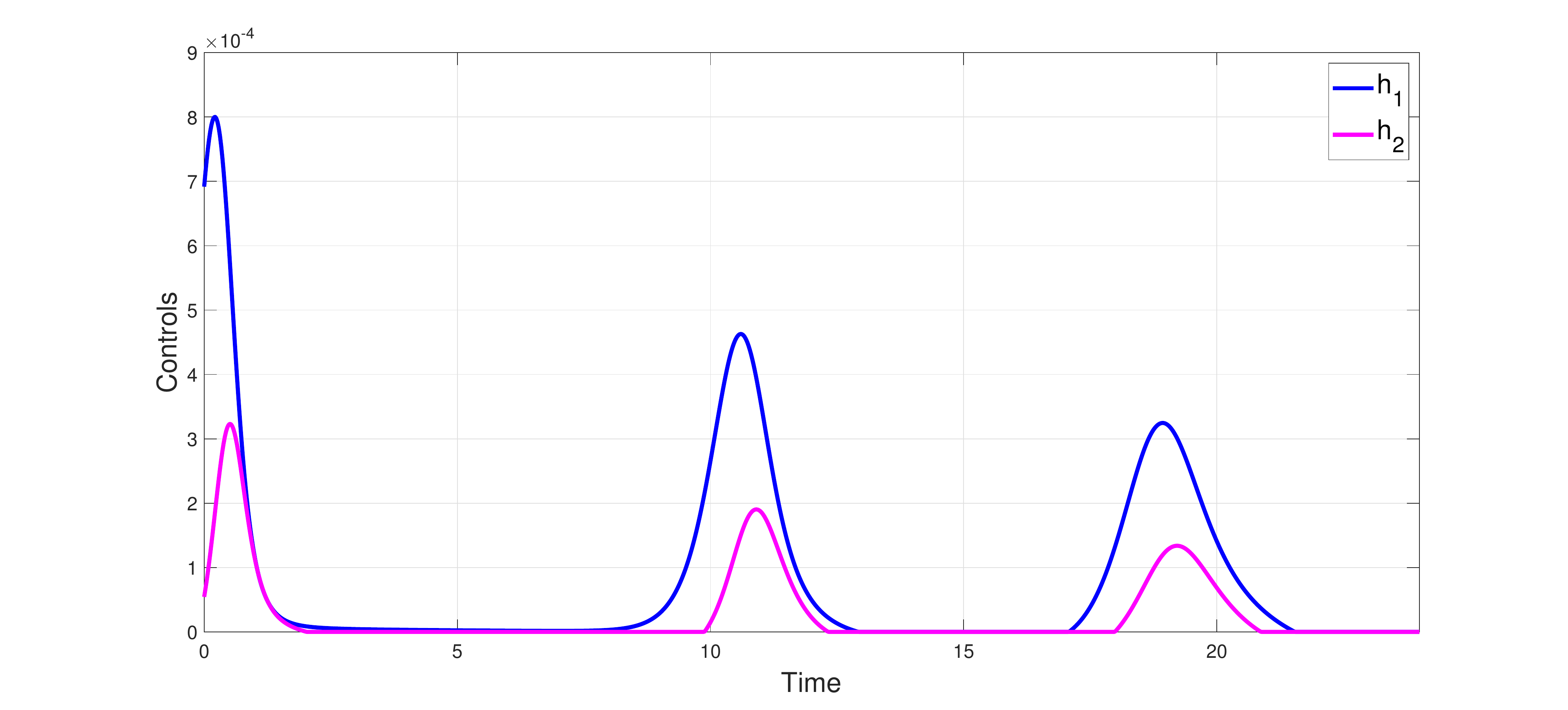}}
\caption{Controls implementation comparison in the Southern European countries when \textit{E. coli} is treated  with gentamicin and amoxicillin.}
\label{fig:control-genta-am}
\end{figure}

\section{Discussion}\label{sec:discussion}

This study aimed to address the pressing global issue of AMR, which is a significant threat to public health. A simple deterministic mathematical model was presented, wherein sensitive and resistant bacteria interacted in the environment, and the presence of MGEs was dependent on resistant bacteria. The qualitative properties of the model were thoroughly analyzed, leading to the proposal of an optimal control problem that emphasized the importance of avoiding mutations and HGT as primary control strategies.
\par 
Furthermore, a case study was conducted using data from the European Antimicrobial Resistance Surveillance Network (EARS-Net), focusing on the resistance and MDR percentages of \textit{Escherichia coli} to gentamicin and amoxicillin in northern, central, and southern Europe. Theoretical results and numerical experiments provided insights into the challenging nature of controlling the spread of resistance in southern European regions through the administration of amoxicillin. These findings highlight the necessity of considering the crucial role of the host immune system in the control of resistance.
\par 
When a load of sensitive and resistant bacteria was simulated for three different regions of Europe, it became apparent that the administration of amoxicillin resulted in higher densities of both bacterial populations than the administration of gentamicin. Furthermore, a noticeable disparity emerged among the three regions of Europe, with southern and central countries being more heavily affected in terms of the prevalence of resistant bacteria, in contrast to northern countries.
\par 
In the case of amoxicillin supply in southern European countries, as antibiotic administration increased, the densities of both sensitive and resistant bacteria increased, with subtle oscillations observed in their behaviour. Conversely, reducing the frequency of antibiotic administration led to a decrease in the populations of both types of bacteria.
\par 
When the influence of the elimination rate of amoxicillin by the host immune system was investigated, focusing on three categories of hosts--those with a robust immune system (RISH), those with a compromised immune system (CISH), and those with a significantly compromised immune system (SCISH)--it was found that when the host immune system was significantly compromised, sensitive bacteria exhibited a slight oscillatory behaviour. Furthermore,it has been observed that individuals with compromised immune systems have played a prominent role in escalating susceptible bacterial burden over time, while individuals with a significantly compromised immune system have exhibited a more substantial impact on the progression of resistant bacterial burden over time.
\par 
When the variables $h_1$ and $h_2$ were incorporated in southern European countries during the treatment of bacteria with gentamicin and amoxicillin, it was found that the implementation of two control strategies effectively reduced both populations of bacteria, namely sensitive and resistant. Specially, the population of sensitive bacteria showed a rapid decline and an oscillatory dynamics, leading to their near elimination. Notably, the control variable $h_2$, responsible for managing mutations through HGT, exerted the most significant effect on controlling the sensitive bacterial population. In contrast, controlling resistant bacteria proved to be more challenging with both antibiotics, although their populations decreased over time. Intriguingly, $h_1$, representing the control of mutations acquired through bacterial exposure to antibiotics, exhibited the highest effort to control the resistant bacterial population. These findings highlight the differential contributions of the control variables $h_1$ and $h_2$ to the management of sensitive and resistant bacterial populations.
\par 
This study sheds light on the dynamics of sensitive and resistant bacterial populations in response to different antibiotics and immune system conditions in Europe. These findings demonstrate that amoxicillin administration leads to higher densities of both sensitive and resistant bacteria compared to when gentamicin was administered. Moreover, a significant regional disparity has emerged, with southern and central countries bearing a more significant burden of resistant bacteria than their northern counterparts. This study provided valuable insights into the dynamics of bacterial populations in response to different antibiotics, immune system conditions, and control strategies. These findings contribute to our understanding of the challenges associated with managing bacterial resistance and highlight the importance of tailored approaches based on specific characteristics of the bacterial population and treatment conditions.
\par 
In conclusion, the scope of this study primarily focused on the interaction between sensitive and resistant bacteria, control strategies to mitigate resistance, and the influence of immune system conditions. It provides insights into the specific case of amoxicillin and gentamicin administration in European countries, highlighting regional differences in resistance burdens. However, this study has some limitations. First, it focuses solely on European countries, which may differ from the global AMR landscape. The findings and observations may not be applicable to regions outside Europe, where different factors and dynamics could influence the spread and control of resistance. Second,  this study mainly explored the dynamics of the \textit{E. coli} population, and the conclusions may not be generalizable to other bacterial species. Third,  while the study examined the impact of antibiotic administration and immune system conditions on bacterial populations, it did not consider other factors that could influence AMR, such as antibiotic usage patterns, infection control practices, or genetic variations in the bacterial strains. These additional factors could provide a more comprehensive understanding of the complex AMR dynamics. Further research and interdisciplinary efforts are warranted to develop effective strategies to combat antimicrobial resistance and to safeguard public health.

\section*{Acknowledgements}

A. Peterson and P.  Aguirre appreciate the support of Proyecto UTFSM PILI1906 and Proyecto Basal CMM Universidad de Chile. J. Romero, K. Acharya and B. Nasri appreciate the support provided by the One Health Modelling Network for Emerging Infections (OMNI-RÉUNIS), which is financially supported by the Natural Sciences and Engineering Research Council of Canada (NSERC) and the Public Health Agency of Canada (PHAC).

\printcredits

\bibliographystyle{cas-model2-names}
\bibliography{cas-refs.bib}

%
\appendix 

\section{Appendices}
\subsection{Proof of Lemma \ref{lemma1}}\label{appendix:A1}
To confirm that $\Omega$ functions as a trapping region, it is necessary to prove that the vector field, which is described by the right-hand side of System \eqref{model_res}, consistently points towards $\Omega$ at all points along the boundary $\partial\Omega$. The portion of \eqref{model_res} that lies perpendicular to the green boundary depicted in Figure \ref{fig_invarianta} is
\begin{equation}\nonumber
       (f_1(0,y),f_2(0,y)) \cdot (1,0) = f_1(0,y)= 0,
\end{equation}
where $0\leq y\leq 1$.
As a consequence, the green boundary located on the $y$-axis maintains its position, making it invariant. Consequently, no trajectory of \eqref{model_res} can cross over into the second quadrant by passing through the green boundary. Furthermore, the point $(0,0)$ is an equilibrium point, and at $(0,1)$, the vector field exhibits the following behavior:
$$ (f_1(0,1),f_2(0,1)) =(0,-\gamma),$$
and, hence, the corresponding orbit remains in the green boundary for $t>0$.
\noindent 

Analogously, the component of \eqref{model_res} in the direction orthogonal to the blue boundary in Figure \ref{fig_invarianta} is
 \begin{equation}\nonumber
       (f_1(x,0),f_2(x,0)) \cdot (0,1) = f_2(x,0)= (1-h_1)qx  \geqslant 0.
 \end{equation}
Since, by definition  $(1-h_1)\geqslant 0$, $q> 0$ and $0<x<1$. If $h_1>0$ the vector field \eqref{model_res} points toward the interior of $\Omega$ along points in the open blue boundary. It follows that orbits are always entering $\Omega$ along the set $y=0$ with $0<x<1$.
In the special case $h_1=0$ the entire $x$-axis is invariant and, hence, no orbit starting in $\Omega$ can cross it towards the $y<0$ half-plane. 
\par
Similarly, orbits are always entering $\Omega$ along the open black boundary line since we have
\begin{equation}\nonumber
       (f_1(x,1-x),f_2(x,1-x)) \cdot (-1,-1)= -f_1(x,1-x) - f_2(x,1-x) =x \alpha + \gamma > 0,
\end{equation}
for $0<x<1$.

Finally, at the point $(1,0)$ the vector field 
$$ (f_1(1,0),f_2(1,0)) =\left(-(\alpha+\gamma)-(1-h_1)q,(1-h_1)q\right)$$
points towards the interior of $\Omega$ since
$$-1<\dfrac{(1-h_1)q}{-(\alpha+\gamma)-(1-h_1)q}\leq0,$$
and, hence, the corresponding orbit remains in $\Omega$ for $t>0$.

Therefore, $\Omega$ is a closed, connected set such that orbits starting in $\partial\Omega$ remain in $\Omega$ for every $t>0$.

\subsection{Proof of the condition $p(y_{max})<0$} \label{appendix:A2}

\begin{equation*}\label{eq1}
\begin{array}{rcl}
p(y_{max})&=&-a_2y_{max}^2+a_1y_{max}+a_0  \\ \\
          &=&-\left[ h_{r2}(R_s+h_s)-h_sR_r\right]\left(\dfrac{R_s-1}{h_s+R_s}  \right)^2  \\  &&+\left[h_{r2}(R_s-1)-h_{r1}q(R_s+h_s)+R_r-R_s  \right]\left(\dfrac{R_s-1}{h_s+R_s}  \right)+h_{r1}q(R_s-1)  \\ \\
          &=& \dfrac{-h_{r2}(R_s-1)^{2}}{h_s+R_s}  +h_sR_r\left(\dfrac{R_s-1}{h_s+R_s}  \right)^2
          +\dfrac{h_{r2}(R_s-1)^{2}}{h_s+R_s}-h_{r1}q(R_s -1)\\
          &&+\dfrac{(R_r - R_s)(R_s-1)}{h_s+R_s}+h_{r1}q(R_s-1) \\ \\
          &=&\dfrac{R_s-1}{h_s+R_s}\left( h_sR_r\dfrac{R_s-1}{h_s+R_s}  +R_r - R_s \right)\\ \\
          &=&\dfrac{R_s-1}{(h_s+R_s)^{2}} \left( h_sR_rR_s-h_{s}R_r   +(R_r - R_s)(h_s+ R_s) \right)\\ \\
          &=&\dfrac{R_s-1}{(h_s+R_s)^{2}} \left( h_sR_rR_s-h_{s}R_r   +R_rh_s - R_sh_s  +R_r R_s -R_s^{2} \right)\\ \\
          &=&\dfrac{R_s-1}{(h_s+R_s)^{2}} \left( h_sR_rR_s - R_sh_s  +R_r R_s -R_s^{2} \right).
\end{array}
\end{equation*}
If $R_r (h_s+1)< h_s +R_s$, a simple computation reveals that $p(y_{max})<0$ since $\dfrac{R_s-1}{(h_s+R_s)^{2}}>0$.

\subsection{Proof of Theorem \ref{thm:teocontrol}} \label{appendix:A3}

When applied to \eqref{model}, Pontryagin's Principle ensures the existence of a vector composed of adjoint variables $\bm \lambda$, with components that fulfill the following condition:
	\begin{align*}
	\dot{\lambda}_1 &= -\frac{\partial H}{\partial S}, \quad \lambda_1(T)=0,  \\
	\dot{\lambda}_2 &= -\frac{\partial H}{\partial R}, \quad \lambda_2(T)=0, \\
	H &= \max_{h_i\in \mathcal{U}}H.
	\end{align*}
	Substituting the derivatives of $H$ with respect to $S$ and $R$ in the previous equations we obtain System~\eqref{adjoint_system}. The optimality conditions for the Hamiltonian are given by
	$
	\dfrac{\partial H}{\partial \textbf{h}^\star}= 0.
	$
  This implies that
	\begin{align*}
h_1&=\frac{-w_1-\bar q\Lambda S(\lambda_1-\lambda_2)}{2w_2} \\
h_2&=\frac{-b_1-aRS(\lambda_1-\lambda_2)}{2b_2}.
    \end{align*}

	Consequently, the optimal controls $h_1^*$ and $h_2^*$ are given by \eqref{optimal_controls},
which completes the proof.
\end{document}